\documentclass[11pt]{amsart}
\pdfoutput=1

\usepackage{mathptmx}
\usepackage{amssymb, amsfonts, amsthm}
\usepackage{tikz}
\usepackage{pgf}
\usepackage{epstopdf}
\usepackage{booktabs}
\usepackage{multicol}
\usepackage{placeins}
\usepackage{caption} 
\usepackage{graphicx}
\usepackage{mathdots}
\usepackage[capitalise]{cleveref}
\usepackage{mathrsfs}
\usepackage{txfonts}
\usepackage{amsopn}
\usepackage{subfig}
\newtheorem{theorem}{Theorem}
\newtheorem{definition}{Definition}
\newtheorem{lemma}{Lemma}

\begin{document}
	
\title[Nonlinear Double Degenerate Parabolic Equations ]{Evolution of Interfaces for the Nonlinear Double Degenerate Parabolic Equation of Turbulent Filtration with Absorption}
	 \thanks{Department of Mathematics, Florida Institute of Technology, Melbourne, FL 32901}

	\author[U. G. Abdulla]{Ugur G. Abdulla}
	\address{Department of Mathematics, Florida Institute of Technology, Melbourne, FL 32901}
	\email{abdulla@fit.edu}	
	
	\author[J. Du]{Jian Du}	
           \address{Department of Mathematics, Florida Institute of Technology, Melbourne, FL 32901}
	\email{jdu@fit.edu}	
		
	\author[A. Prinkey]{Adam Prinkey}

           \author[C. Ondracek]{Chloe Ondracek} 		
				
	\author[S. Parimoo]{Suneil Parimoo}

	\begin{abstract}
 We prove the short-time asymptotic formula for the interfaces and local solutions near the interfaces for the nonlinear double degenerate reaction-diffusion equation of turbulent filtration with strong absorption
\[ u_t=\Big(|(u^{m})_x|^{p-1}(u^{m})_x\Big)_x-bu^{\beta}, \, mp>1, \, \beta >0. \]
Full classification is pursued in terms of the nonlinearity parameters  $m, p,\beta$ and asymptotics of the initial function near its support. Numerical analysis using a weighted essentially nonoscillatory (WENO) scheme with interface capturing is implemented, and comparison of numerical and analytical results is presented.
\end{abstract}

	\maketitle

	\section{Introduction}
\label{sec:intro}

Consider the Cauchy problem (CP) for the nonlinear double degenerate parabolic equation: 
\begin{equation}\label{OP1}
Lu\equiv u_t-\Big(|(u^{m})_x|^{p-1}(u^{m})_x\Big)_x+bu^{\beta} = 0, \ x\in \mathbb{R}, ~0<t<T,
\end{equation}
\begin{equation}\label{IF1}
u(x,0)=u_0(x),~~x\in \mathbb{R},
\end{equation} 
where $m, p, b, \beta >0, ~mp>1,~0<T\leq +\infty,$ and $u_0$~is nonnegative and continuous. Equation \eqref{OP1} arises in turbulent polytropic filtration of a gas in porous media \cite{Barenblatt1, Est-Vazquez, Leibenson}. Under the condition $mp>1$, the PDE \eqref{OP1} posesses finite speed of propagation property. Assume that $\eta(0)=0$, where $\eta(\cdot)$ be an interface, or free boundary defined as
\[\eta(t):=\text{sup}\{x:u(x,t)>0\}.\]
Let
\begin{equation}\label{IF2}
u_0(x)\sim C(-x)_+^{\alpha},\, \text{as} \, x\rightarrow 0^{-},\, \text{for some} \, C>0,\, \alpha>0.
\end{equation}
The goal of this paper is to present full classification of the short time behavior of the interface $\eta$, and local solution near $\eta$ in terms of parameters $m,p,b, \beta, C,$ and $\alpha$. 
Our estimations will be global in time in the special case when
\begin{equation}\label{IF3}
u_0(x)=C(-x)_+^{\alpha},\, x\in \mathbb{R},
\end{equation}
and the minimal solution to the problem \eqref{OP1}, \eqref{IF3} is of self-similar form.

The initial development of interfaces and structures of local solutions near the interfaces is very well understood in the case of the reaction-diffusion equations with porous medium ($p=1$ in \eqref{OP1}) and $p$-Laplacian ($m=1$ in \eqref{OP1}) type diffusion terms.
Full classification of the evolution of interfaces and the local behaviour of solutions near the interfaces for the reaction-diffusion equations with porous medium type diffusion ($p=1$ in \eqref{OP1})  was presented in \cite{Abdulla1} for the case of slow diffusion ($m>1$), and in \cite{Abdulla3} for the fast diffusion case ($0<m<1$). Similar classification for the reaction-diffusion equations with $p$-Laplacian type diffusion ($m=1$ in \eqref{OP1}) is presented in a recent paper \cite{AbdullaJeli}. 

The organization of the paper is as follows: In \cref{sec: description of the main results} we outline the main results. For clarity of the exposure we describe further technical details of the main results in \cref{sec: details of the main results}. In \cref{sec: preliminary results} we apply nonlinear scaling techniques for some preliminary estimations which are necessary for the proof of main results. In \cref{sec: proofs of the main results.} we prove the main results. Finally, in \cref{sec: numerical solution} we confirm our analytical results from \cref{sec: description of the main results} numerically by implementing a WENO scheme. 
We provide explicit values of some of constants in \cref{sec: appendix}, and numerical graphs in \cref{sec:fig}.\\ 

	\section{Main Results}\label{sec: description of the main results} 
There are four different subcases, as shown in \cref{fig:figure 1}. The main results are the following:
\begin{figure}
\centering
\includegraphics[width = 0.7\textwidth]{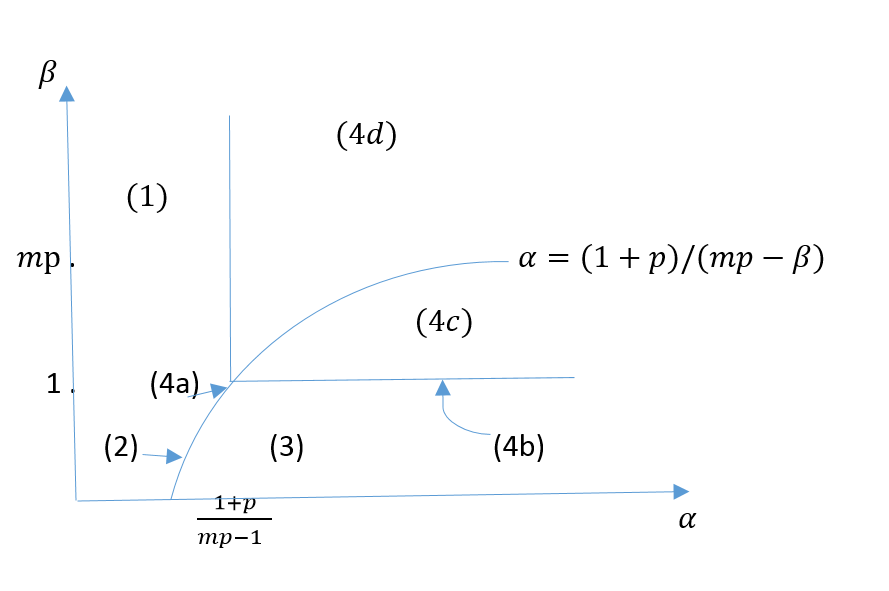}
\caption{Classification of different cases in the $(\alpha, \beta)$-parameter space for interface development in problem \eqref{OP1}-\eqref{IF3}.}
\label{fig:figure 1}
\end{figure}
\begin{theorem}\label{theorem 1}
If $\alpha < \frac{1+p}{mp- \text{min}\{1,\beta\}}$,  then the interface initially expands and
\begin{equation}\label{I1}
\eta(t)\sim \xi_* t^{1/(1+p-\alpha(mp-1))},~\text{as}~t\rightarrow 0^{+},
\end{equation}
where
\begin{equation}\label{XI1}
\xi_*=C^{\frac{mp-1}{1+p-\alpha(mp-1)}}\xi'_*,
\end{equation}
and $\xi'_*,$~is a positive number depending only on~$m$,~$p$, and $\alpha$. For arbitrary $\rho<\xi_*,$ there exists a positive number $f(\rho)$~depending on $C, m, p, $~and~$\alpha$ such that:
\begin{equation}\label{U1}
u(x,t)\sim t^{\alpha/(1+p-\alpha(mp-1))}f(\rho),~\text{as}~t\rightarrow0^{+},
\end{equation}
along the curve $x=\xi_{\rho}(t)=\rho t^{1/(1+p-\alpha(mp-1))}.$ 
\end{theorem}

\begin{theorem}\label{theorem 2} 
Let $0<\beta<1, \alpha=(1+p) /(mp-\beta)$ and 
\[C_*=\Bigg[\frac{b(mp-\beta)^{1+p}}{(m(1+p))^{p}p(m+\beta)}\Bigg]^{\frac{1}{mp-\beta}}.\]
Then interface expands or shrinks according as $C>C_*$ or $C<C_*$ and
\begin{equation}\label{I3}
\eta(t)\sim\zeta_*t^{\frac{mp-\beta}{(1+p)(1-\beta)}},~\text{as}~t\rightarrow 0^{+},
\end{equation}
where $\zeta_* \lessgtr 0$ if $C \lessgtr C_*$, and for arbitrary $\rho < \zeta_*$ there exists $f_1(\rho)>0$ such that:
\begin{equation}\label{U4}
u(x,t)\sim t^{1/(1-\beta)}f_1(\rho),~\text{for}~x=\rho t^{\frac{mp-\beta}{(1+p)(1-\beta)}},~\text{as}~t\rightarrow 0^{+}.
\end{equation}
\end{theorem}

\begin{theorem}\label{theorem 3} 
If $0<\beta<1, \text{and} \, \alpha >(1+p)/(mp-\beta)$, then the interface shrinks and
\begin{equation}\label{I4}
\eta(t)\sim - \ell_*t^{1/\alpha(1-\beta)},~ \text{as}~t\rightarrow 0^{+},
\end{equation}
where~$\ell_*=C^{-1/\alpha}(b(1-\beta))^{1/\alpha(1-\beta)}. $For arbitrary $\ell>\ell_*$ we have:
\begin{equation}\label{U5}
u(x,t)\sim\Big [C^{1-\beta}(-x)_+^{\alpha(1-\beta)}-b(1-\beta)t\Big]^{1/(1-\beta)},\, \text{as} \,t\rightarrow 0^{+},
\end{equation}
along the curve $x=\eta_l(t)=-lt^{1/\alpha(1-\beta)}.$ 
\end{theorem}

\begin{theorem}\label{theorem 4}
If $\alpha \geq (1+p)/(mp-1)$ and $\beta \geq 1$, then interface initially remains stationary.
\end{theorem}
\section{Further Details of the Main Results}\label{sec: details of the main results} 
{\it Further details of \cref{theorem 1}}. 
$f$ is a shape function of the self-similar solution to the problem \eqref{OP1}, \eqref{IF3} with $b=0$:

\begin{equation}\label{U8}
u(x,t)=t^{\frac{\alpha}{1+p-\alpha(mp-1)}}f(\xi), \, \xi=xt^{-\frac{1}{1+p-\alpha(mp-1)}}.
\end{equation}
In fact, $f$ is a solution of the following nonlinear ODE problem in $\mathbb{R}$:
\begin{equation}\label{selfsimilarODE1} 
\big(|(f^{m}(\xi))'|^{p-1}(f^{m}(\xi))'\big)'+(1+p-\alpha (mp-1))^{-1}(\xi f'(\xi)-\alpha f(\xi))=0,
\end{equation}
\begin{equation}\label{selfsimilarODE1bc}
f(\xi)\sim C(-\xi)^\alpha, \text{ as } \xi \downarrow -\infty, ~f(+\infty)= 0.
\end{equation}
Moreover, $\exists \, \xi_*>0$ such that: $f(\xi)\equiv 0$ for $\xi \geq \xi_*$; $f(\xi)>0$ for $\xi<\xi_*$. Dependence on $C$ is given by the following relation:
\begin{subequations}\label{F1}
\begin{align}
f(\rho)=C^{1+p/(1+p-\alpha(mp-1))}f_0\Big(C^{(mp-1)/(\alpha(mp-1)-(1+p)} \rho\Big),\label{a1}\\
f_0(\rho)=w(\rho,1), \quad \xi'_*=\sup\{\rho:f_0(\rho)>0\}>0, \label{a2}
\end{align}
\end{subequations}
where $w$ is a minimal solution of the CP \eqref{OP1}, \eqref{IF3} with $b=0,\, C=1.$ Lower and upper estimations for $f$ are given in \eqref{E9}. We also have that: \\
\begin{equation}\label{XI2}
\xi'_*=A_0^{\frac{mp-1}{1+p}}\Bigg[\frac{(mp)^{p}(1+p-\alpha(mp-1))}{(mp-1)^{p}}\Bigg]^{\frac{1}{1+p}}\xi''_*,
\end{equation}
where $A_0=w(0,1)$ and $\xi''_*$~is some number belonging to the segment $[\xi_1,\xi_2]$ (see ~\cref{sec: appendix}).
In the particular case $\alpha=p(mp-1)^{-1}$ and $mp>1+p-p(\text{min}\{1,\beta\})$, the explicit solution of \eqref{OP1}, \eqref{IF3} with $b=0$ is given by \eqref{U7} and
\begin{equation}\label{XI5}
\xi_1=\xi_2 = 1,~\xi'_*=(mp)^{p}(mp-1)^{-p},~f_0(x)=\big(\xi'_*-x\big)_+^{p/(mp-1)}
\end{equation}

{\it Further details of \cref{theorem 2}}.
If $p(m+\beta)=1+p$, the solution to \eqref{OP1}, \eqref{IF3} is
\begin{equation}\label{U2}
u(x,t)=C(\zeta_*t-x)_+^{\frac{1+p}{mp-\beta}}, \, \zeta_*=b(1-\beta)C^{\beta-1}((C/C_*)^{mp-\beta}-1).
\end{equation}
Let $p(m+\beta)\neq1+p.$~If~$C=C_*$,~then~$u_0$ is a stationary solution to \eqref{OP1}, \eqref{IF3}. If~$C\neq C_*,$ then the minimal solution to \eqref{OP1}, \eqref{IF3} is of the self similar form: 
\begin{equation}\label{U3}
u(x,t)=t^{1/(1-\beta)}f_1(\zeta), \, \zeta=xt^{-\frac{mp-\beta}{(1+p)(1-\beta)}},
\end{equation}
\begin{equation}\label{I2}
\eta(t) =\zeta_* t^{\frac{mp-\beta}{(1+p)(1-\beta)}}, \, 0\leq t<+\infty,
\end{equation}
where $f_1(\zeta)$ solves the following nonlinear ODE problem:
\begin{equation}\label{selfsimilarNODE3}
\big(|(f_1^{m})'|^{p-1}(f_1^{m})'\big)'+\frac{mp-\beta}{(1+p)(1-\beta)}\zeta f_1'-\frac{1}{1-\beta}f_1-bf^{\beta}_1=0,~\zeta\in\mathbb{R},
\end{equation}
\begin{equation}\label{selfsimilarNODE3bc}
f_1(\zeta)\sim C(-\zeta)^{(1+p)/(mp-\beta)}, \text{ as } \zeta 
\downarrow -\infty, \text{ and } f_1(\zeta) \rightarrow 0, \text{ as } \zeta \uparrow +\infty.
\end{equation}
Moreover, $\exists \zeta_*$ such that $f(\zeta)\equiv 0$ for $\zeta \geq \zeta_*$; $f(\zeta)>0$ for $\zeta<\zeta_*$. If $C>C_*$~then the interface expands, $f_1(0)=A_1>0 $~(see \eqref{lemma 3}), and:
\begin{equation}\label{E1}
C_1t^{\frac{1}{1-\beta}}\Big(\zeta_1-\zeta\Big)_+^{\mu}\leq u\leq C_2t^{\frac{1}{1-\beta}}\Big(\zeta_2-\zeta\Big)_+^{\mu}, \, 0\leq x<+\infty, \, 0<t<+\infty,
\end{equation}
where
\begin{equation*}
\begin{cases}
\mu=p(mp-1)^{-1}, & \text{if}~ p(m+\beta)>1+p, \\
\mu=(1+p)(mp-\beta)^{-1}, & \text{if}~ p(m+\beta)<1+p,
\end{cases}
\end{equation*}
which implies:
\begin{equation}\label{E2}
\zeta_1\leq \zeta_*\leq \zeta_{2}.
\end{equation}
If $0<C<C_*$, then the interface shrinks. If $p(m+\beta)>1+p$ then: 
\begin{gather}
[C^{1-\beta}(-x)_+^{\frac{(1+p)(1-\beta)}{mp-\beta}}-b(1-\beta)t]_+^{\frac{1}{1-\beta}}
\leq u \leq\nonumber\\
[C^{1-\beta}(-x)_+^{\frac{(1+p)(1-\beta)}{mp-\beta}}-b(1-\beta)(1-(C/C_*)^{mp-\beta})t]_+^{\frac{1}{1-\beta}},
 \, x\in \mathbb{R},\, t\geq 0,\label{E3}
\end{gather}
which also implies \eqref{E2}, where we replace $\zeta_1$ (respectively, $\zeta_2$) with respective negative values given in~\cref{sec: appendix}.
However, if  $p(m+\beta)<1+p,$ then:
\begin{equation}\label{E4}
C_*\Big(-\zeta_3t^{\frac{mp-\beta}{(1+p)(1-\beta)}}-x\Big)_+^{\frac{1+p}{mp-\beta}}\leq u \leq C_3\Big(-\zeta_4t^{\frac{mp-\beta}{(1+p)(1-\beta)}}-x\Big)_+^{\frac{1+p}{mp-\beta}},\, x\in \mathbb{R},\, t>0,
\end{equation}
where the left-hand side is valid for $x\geq -\ell_0t^{\frac{mp-\beta}{(1+p)(1-\beta)}},$ while the right-hand side is valid for $x\geq-\ell_1t^{\frac{mp-\beta}{(1+p)(1-\beta)}}$. From \eqref{E4}, \eqref{E2} follows if we replace $\zeta_1$ and $\zeta_2$ with $-\zeta_3$ and $-\zeta_4$, respectively.

{\it Further Details of \cref{theorem 4}}. There are four different subcases (see \cref{fig:figure 1}).

(4a) If $\beta=1, \alpha =(1+p)/(mp-1)$, the unique minimal solution to \eqref{OP1}, \eqref{IF3} is
\begin{equation}\label{U6}
u_C=C(-x)_+^{(1+p)/(mp-1)}e^{-bt}\big[1-(C/\bar C)^{mp-1}b^{-1}(1-e^{-b(mp-1)t})\big]^{1/(1-mp)}
\end{equation}
where
\begin{align}
T=+\infty,\quad \text{if}~~b\geq (C/\bar C)^{mp-1},\nonumber\\
T=(b(1-mp))^{-1}\text{ln}[1-b(\bar C/C)^{mp-1}],\quad \text{if}~-\infty<b<(C/\bar C)^{mp-1},\nonumber\
\end{align}

(4b) Let $\beta=1~ \text{and}~ \alpha>(1+p)/(mp-1). ~$Then $\forall \epsilon>0 \ \exists~x_{\epsilon}<0$ and $\delta_{\epsilon}>0$ such that:
\begin{gather} 
(C-\epsilon)(-x)_+^{\alpha}e^{-bt}\leq u 
\leq (C+\epsilon)(-x)_+^{\alpha}e^{-bt}\nonumber\\\big[1-\epsilon (b(mp-1))^{-1}\big(1-e^{-b(mp-1)t}\big)\big]^{1/(1-mp)}, 
~x>x_{\epsilon},~~0\leq t\leq \delta_{\epsilon}. \label{E5}
\end{gather}

(4c) Let $1<\beta<mp ~\text{and}~ \alpha\geq (1+p)/(mp-\beta)$. Then $\forall \epsilon>0 \ \exists~x_{\epsilon}<0$ and $\delta_{\epsilon}>0$ such that:
\begin{equation}\label{E6}
g_{-\epsilon}(x,t) \leq u(x,t) \leq g_{\epsilon}(x,t), ~x\geq x_{\epsilon},~0\leq t \leq \delta_{\epsilon}
\end{equation}
where
\[g_{\epsilon}(x,t)=\begin{cases}
[(C+\epsilon)^{1-\beta}|x|^{\alpha(1-\beta)}+b(\beta-1)(1-\epsilon-\kappa_{\epsilon})t]^{1/(1-\beta)}, &  \text{if}~ x_{\epsilon}\leq x <0,\\
0, & \text{if}~ x\geq 0,
\end{cases}\]
where $\kappa_{\epsilon} = 0$, if $\alpha > (1+p)/(mp-\beta)$;~$\kappa_\epsilon=((C+\epsilon)/C_*)^{mp-\beta}$, if~$\alpha = (1+p)/(mp-\beta)$.

(4d) Let either $1<\beta<mp,~(1+p)/(mp-1)\leq \alpha<(1+p)/(mp-\beta),$ or $\beta \geq mp,~\alpha \geq (1+p)/(mp-1).$ If $\alpha=(1+p)/(mp-1)$ then for $\forall \epsilon>0~\exists~x_{\epsilon}<0$ and $\delta_{\epsilon}>0$ such that:
\begin{equation}\label{E7}
(C-\epsilon)(-x)_+^{\frac{1+p}{mp-1}}(1-\gamma_{-\epsilon}t)^{\frac{1}{1-mp}} \leq u 
\leq (C+\epsilon)(-x)_+^{\frac{1+p}{mp-1}}(1-\gamma_{\epsilon}t)^{\frac{1}{1-mp}}
\end{equation}
for $x>x_{\epsilon},~0\leq t\leq \delta_{\epsilon}$ (see \cref{sec: appendix} for $\gamma_{\epsilon}$).
However, if $\alpha > (1+p)/(mp-1)$, then for arbitrary sufficiently small $\epsilon>0$, there exist $x_{\epsilon}<0$ and $\delta_{\epsilon}>0$ such that:
\begin{equation}\label{E8}
(C-\epsilon)(-x)_+^{\alpha} \leq u \leq (C+\epsilon)(-x)_+^{\alpha}(1-{\epsilon}t)^{1/(1-mp)},~x>x_{\epsilon},~0\leq t\leq \delta_{\epsilon}.
\end{equation}

{\it Results for the case $b=0$}.  

(1) If $\alpha = p/(mp-1)$ the minimal solution to the problem \eqref{OP1}, \eqref{IF3} is
\begin{equation}\label{U7}
u(x,t)=C(\xi_*t-x)_+^{p/(mp-1)}, \, \xi_*=C^{mp-1}\Big(\frac{mp}{mp-1}\Big)^{p}.
\end{equation}
If $0<\alpha <(1+p)/(mp-1),$ then the minimal solution to \eqref{OP1}, \eqref{IF3} has the self-similar form \eqref{U8} and
\begin{equation}\label{I5}
\eta(t)=\xi_* t^{\frac{1}{1+p-\alpha(mp-1)}}, ~~~0\leq t<+\infty,
\end{equation}
where $\xi_*$ and $f$ solve \eqref{selfsimilarODE1}-\eqref{selfsimilarODE1bc}.  We have the estimation
\begin{equation}\label{E9}
C_4t^{\frac{\alpha}{1+p-\alpha(mp-1)}}(\xi_3-\xi)_+^{\frac{p}{mp-1}}\leq u\leq C_5t^{\frac{\alpha}{1+p-\alpha(mp-1)}}(\xi_4-\xi)_+^{\frac{p}{mp-1}},
~x\geq 0,~t\geq 0,
\end{equation}
(see \cref{sec: appendix}). If $\alpha=p/(mp-1)$, then $\xi_3=\xi_4=\xi_*$ and both lower and upper estimations in \eqref{E9} coincide with the solution \eqref{U7}. 

(2) If $\alpha =(1+p)/(mp-1)$, then interface initially remains stationary. Explicit solution to \eqref{OP1}, \eqref{IF3} is
\begin{equation}\label{U9}
u_C(x,t)=C(-x)_+^{(1+p)/(mp-1)}[{\lambda}(t_* -t)(1-mp)]^{1/(1-mp)}, \, x\in \mathbb{R}, \, 0\leq t<t_*,
\end{equation}
where
\[t_*=1/\lambda(1-mp), \text{with} ~\lambda = -C^{mp-1}\frac{p(m+1)(m(1+p))^{p}}{(mp-1)^{1+p}}.\]

(3) If $\alpha >(1+p)/(mp-1)$, then interface again remain stationary, and for $\forall \epsilon>0~\exists~x_{\epsilon}<0$ and $\delta_{\epsilon}>0$ such that: 
\begin{equation}\label{E10}
(C-\epsilon)(-x)_+^{\alpha}\leq u \leq (C+\epsilon)(-x)_+^{\alpha}(1-\epsilon t)^{1/(1-mp)},~ x_{\epsilon}\leq x, ~0\leq t\leq \delta_{\epsilon}.
\end{equation}

	\section{Preliminary Results}\label{sec: preliminary results} 
The prelude of the mathematical theory of the nonlinear degenerate parabolic equations is the papers
 \cite{zeldovich,Barenblatt1} (see also \cite{Barenblatt2}), where instantaneous point source type particular solutions were constructed and analyzed. The property of finite speed of propagation and the existence of compactly supported nonclassical solutions and interfaces became a motivating force of the general theory.
Mathematical theory of nonlinear degenerate parabolic equations began with the paper \cite{Oleinik} on the porous medium equation (\eqref{OP1} with $p=1$). Currently there is a well established general theory of the nonlinear degenerate parabolic equations (see \cite{Vazquez2,Dibe-Sv,Abdulla6,Abdulla7,Abdulla8,shmarev}. Boundary value problems for \eqref{OP1} been have been investigated in \cite{Kalashnikov3,Est-Vazquez,Tsutsumi,Ishige, Ivanov1, Ivanov2}.

\begin{definition}[Strong Solution]\label{def: weak soln} (\cite{Est-Vazquez,Tsutsumi}) Let $u_0\in L^1(\mathbb{R})$ and nonnegative. 
A measurable nonnegative function $u(x,t)$ defined in $\mathbb{R} \times (0,T)$ is a strong solution of \eqref{OP1}, \eqref{IF1} if 
\begin{equation}\label{strongsolution}\begin{cases}
u\in C([0,T); L^1(\mathbb{R})\cap L^\infty([\delta,T)\times \mathbb{R}), \ \forall \delta>0,\\
u^m(t,\cdot)\in W^{1,\infty}(\mathbb{R}), \ a.e. \ 0<t<T,\\
u_t, (|(u^m)_x|^{p-1}(u^m)_x)_x \in L^1_{loc}(0,T; L^1(\mathbb{R})),
\end{cases}
\end{equation}
and $u$ satisfies  \eqref{OP1}, \eqref{IF1} for a.e. $0<t<T, x\in \mathbb{R}$.
\end{definition}
Existence, uniqueness and comparison theorems for the strong solution of the CP  \eqref{OP1}, \eqref{IF1} was proved in \cite{Est-Vazquez} for the case $b=0$, and in \cite{Tsutsumi} for $b> 0$. In \cite{Est-Vazquez} it is proved that the strong solution of \eqref{OP1}, $b=0$, is locally H\"{o}lder continuous. Local H\"{o}lder continuity of the locally bounded weak solutions (accordingly, strong solutions) of the general second order multidimensional nonlinear degenerate parabolic equations with double degenerate diffusion term is proved in \cite{Ivanov1,Ivanov2}. The following is the standard comparison result, which is widely used throughout the paper: 
\begin{lemma}\label{CT}
Let $g$ be a non-negative amd continuous function in $\overline{Q}$, where:
\[Q = \{(x,t): \eta_0(t) < x < +\infty, \, 0 < t  < T \leq +\infty \}\]
$g = g(x,t)$ is in $C^{2,1}_{x,t}$ in $Q$ outside a finite number of curves: $x = \eta_j(t)$, which divide $Q$ into a finite number of subdomains: $Q^j$, where $\eta_j \in C[0,T]$; for arbitrary $\delta > 0$ and finite $\Delta \in (\delta, T]$ the function $\eta_j$ is absolutely continuous in $[\delta, \Delta]$. Let $g$ satisfy the inequality:
\[ Lg\equiv g_t-\Big(|(g^{m})_x|^{p-1}(g^{m})_x\Big)_x+bg^{\beta} \geq 0, (\leq 0),\]
at the points of $Q$ where $g \in C^{2,1}_{x,t}$. Assume also that the function: $|(g^{m})_x|^{p-1}(g^{m})_x$ is continuous in $Q$ and $g \in L^{\infty}(Q \cap (t \leq T_1))$ for any finite $T_1 \in (0,T]$. If in addition we have that:
\[g(\eta_0(t),t) \geq  (\leq) \, u(\eta_0(t),t), \,\, g(x,0) \geq (\leq) \, u(x,0), \]
then
\[g \geq (\leq) \, u, \, \, \text{in} \, \, \overline{Q}\]
\end{lemma}

Suppose that $u_0\in L^1_{loc}(\mathbb{R})$, and may have unbounded growth as $|x|\rightarrow +\infty$. It is well known that in this case some restriction must be imposed on the growth rate for existence, uniqueness of the solution to the CP \eqref{OP1}, \eqref{IF1}. For the particular cases of the equation \eqref{OP1} with $b=0$ this question was settled down in \cite{BCP,HerreroPierre} for the porous medium equation ($p=1$) with slow ($m>1$) and fast ($0<m<1$) diffusion; and in \cite{DibeHerrero1,DibeHerrero2} for the $p$-Laplacian equation ($m=1$) with slow ($p>1$) and fast ($0<p<1$) diffusion; The case of reaction-diffusion equation $m>1,p=1,b > 0$ is analyzed in \cite{Kalashnikov4,KPV,Abdullaev1}. Surprisingly, only a partial result is available for the double-degenerate PDE \eqref{OP1}. The sharp sufficient condition for the existence of the solution to the CP for \eqref{OP1}, $b=0$  
is established in \cite{Ishige}. In particular, it follows from \cite{Ishige} that the CP \eqref{OP1},\eqref{IF3} has a solution
if and only if $\alpha\leq(1+p)/(mp-1)$. Moreover, solution is global ($T=+\infty$) if $\alpha<(1+p)/(mp-1)$ and only local in time if $\alpha=(1+p)/(mp-1)$. Uniqueness of the solution is an open problem. For our purposes it is satisfactory to employ the notion of the minimal solution.
\begin{definition}[Minimal Solution]\label{def: minimal soln} Let $u_0\in L_{loc}^1(\mathbb{R})$ and nonnegative. Nonnegative weak solution of the CP  \eqref{OP1}, \eqref{IF1} is called a {\it minimal solution} if
\begin{equation}\label{minimalsolution}
0\leq u(x,t)\leq v(x,t),
\end{equation}
for any nonnegative weak solution $v$ of the same problem  \eqref{OP1}, \eqref{IF1}.
\end{definition}
Note that the minimal solution is unique by definition. The following standard comparison result is true in the class of minimal solutions:
\begin{lemma}\label{CT2} Let $u$ and $v$ be minimal solutions of the CP  \eqref{OP1}, \eqref{IF1}. If
\begin{equation*}
u(x,0)\geq (\leq) \ v(x,0), \ x\in \mathbb{R},
\end{equation*}
then
\begin{equation*}
u(x,t)\geq (\leq) \ v(x,t), \ (x,t)\in \mathbb{R}\times (0,T).
\end{equation*}
\end{lemma}

If the function $u(x,t)$ is a minimal solution to CP \eqref{OP1}, \eqref{IF3} with $b=0$, then the function:
\[\bar u(x,t)=\exp(-bt)u(x,(b(1-mp))^{-1}(\exp(-b(mp-1))t-1)),\] 
is a minimal solution to \eqref{OP1} with $ b\neq 0~\text{and}~ \beta=1$. Hence, from the above mentioned result it follows that the unique minimal solution to CP \eqref{OP1}, \eqref{IF3} with $mp>1, b>0, ~\beta =1, \text{and}~ \alpha=(1+p)/(mp-1)$, is the function $\bar u_C(x,t)$ from \eqref{U6}.

In the following lemmas  we establish some preliminary estimations of the solution to the CP.

\begin{lemma}\label{lemma 1}
 If $b=0, 0 < \alpha < (1+p)/(mp-1)$, then the minimal solution $u$ of the CP \eqref{OP1}, \eqref{IF3} has a self-similar form \eqref{U8}, where the self-similarity function $f$ satisfies \eqref{F1}. If $u_0$ satisfies \eqref{IF2}, then the solution to CP \eqref{OP1}, \eqref{IF1} satisfies \eqref{I1}-\eqref{U1}.
\end{lemma}

\begin{lemma}\label{lemma 2}
If $0<\alpha<\frac{1+p}{mp-\min(1, \beta)}$, then solution to the CP \eqref{OP1}-\eqref{IF2} satisfies \eqref{U1}.
\end{lemma}
In the next lemma we analyze special class of finite travelling wave solutions. By a finite travelling-wave solution with velocity $0 \neq k \in \mathbb{R}$ we mean a solution $u(x,t) = \phi((kt-x))$, where $\phi(y) \geq 0$, $\phi \not\equiv 0$, and $\phi(y) = 0$ for $y \leq y_0$ for some $y_0 \in \mathbb{R}$. 
\begin{lemma}\label{lemma tw}
Let $0<\beta <1, p(m+\beta) > 1+p$. PDE \eqref{OP1} admits a finite travelling-wave solution, $u(x,t) = \phi((kt-x))$, with $\phi(y) = 0$ for $y\leq 0$; $\phi(y)>0$, for $y>0$, and:
\begin{align}\label{LIM1TW}
\lim_{y \to +\infty} y^{-\frac{1+p}{mp-\beta}}\phi(y) = C_*. 
\end{align}
\end{lemma}
\begin{lemma}\label{lemma 3}
If $0<\beta<1, \alpha=(1+p)/(mp-\beta)$, then the minimal solution $u$ to the CP \eqref{OP1}, \eqref{IF3} has a self-similar form \eqref{U3}, where the self-similarity function $f$ satisfies \eqref{selfsimilarNODE3},\eqref{selfsimilarNODE3bc}. If $C>C_*$, then $f_1(0)=A_1>0$, where $A_1$ depends on $m, p, \beta, C$ and $b$. If $u_0$ satisfies \eqref{IF2} with $\alpha=(1+p)/(mp-\beta) \text{and} ~C>C_*$, then the solution to CP \eqref{OP1}, \eqref{IF1} satisfies: 
\begin{equation}\label{U10}
u(0,t)\sim A_1 t^{1/(1-\beta)}, ~\text{as}~ t\rightarrow 0^+
\end{equation}
\end{lemma}

\begin{lemma}\label{lemma 4}
If $0<\beta<1, \text{and}~ \alpha>(1+p)/(mp-\beta)$, then the solution $u$ to the CP \eqref{OP1}-\eqref{IF2} satisfies \eqref{U5}. 
\end{lemma}

\begin{proof}[Proof of \cref{lemma 1}] 
Let $u$ be a unique minimal solution of the problem  \eqref{OP1}, \eqref{IF3}. If we consider a function:
\begin{equation}\label{U11}
u_k(x,t)=ku\big(k^{-1/\alpha}x,k^{(\alpha(mp-1)-(1+p))/\alpha}t\big),~ k>0,
\end{equation}
it may easily be checked that this satisfies \eqref{OP1}, \eqref{IF3}. Since $u$ is a minimal solution we have: 
\begin{equation}\label{U12}
u(x,t)\leq ku\big(k^{-1/\alpha}x,k^{(\alpha(mp-1)-(1+p))/\alpha}t\big),~k>0.
\end{equation}
By changing the variable in \eqref{U12} as
\begin{equation}\label{changeofvar}
y= k^{-1/\alpha}x, \ \tau=k^{(\alpha(mp-1)-(1+p))/\alpha}t,
\end{equation}
we derive \eqref{U12} with $k$ replaced with $k^{-1}$. Since $k>0$ is arbitrary, \eqref{U12} follows with "=". 
If we choose $k=t^{\alpha/(1+p-\alpha(mp-1))},$ the latter implies \eqref{U8} with $f(\xi)=u(\xi,1)$, where $f$ is a nonnegative and continuous solution of \eqref{selfsimilarODE1},\eqref{selfsimilarODE1bc}. By \cite{Barenblatt1}, PDE
\eqref{OP1} has a finite speed of propagation property, and minimal solution of \eqref{OP1}, \eqref{IF3} has an expanding interface. Therefore, the upper bound $\xi_*$ of the support of $f$ is positive and finite;  $f$ is positive and smooth for $\xi<\xi_*$ and $f=0$ for $\xi\geq \xi_*.$ Thus, \eqref{I5} is valid. Proof of \eqref{XI1} and \eqref{F1} coincide with the proof given in Lemma 3.2 of \cite{Abdulla1}.

Now suppose that $u_0$ satisfies \eqref{IF2}. Then for arbitrary sufficiently small $\epsilon>0$, there exists an $x_{\epsilon}<0$ such that:
\begin{equation}\label{E11}
(C-\epsilon)(-x)_+^{\alpha}\leq u_0(x) \leq (C+\epsilon)(-x)_+^{\alpha},~x\geq x_{\epsilon}.
\end{equation}
Let $u_{\epsilon}(x,t)$ (respectively, $u_{-\epsilon}(x,t)$) be a minimal solution to the CP \eqref{OP1}, \eqref{IF1} with initial data $(C+\epsilon)(-x)_+^{\alpha}$ (respectively, $(C-\epsilon)(-x)_+^{\alpha}$). Since the solution to the CP \eqref{OP1}, \eqref{IF1} is continuous, there exists a number $\delta=\delta(\epsilon)>0$ such that:
\begin{equation}\label{E12}
u_{\epsilon}(x_{\epsilon},t)\geq u(x_{\epsilon},t),~u_{-\epsilon}(x_{\epsilon},t)\leq u(x_{\epsilon},t),~\text{for}~0\leq t \leq \delta.
\end{equation}
From \eqref{E11}, \eqref{E12}, and by applying the comparison result, \eqref{CT}, it follows that:
\begin{equation}\label{E13}
u_{-\epsilon}\leq u \leq u_{\epsilon},~\text{for}~x\geq x_{\epsilon},~0\leq t\leq \delta. 
\end{equation}
We have:
\begin{equation}\label{U14}
u_{\pm \epsilon}(\xi_{\rho}(t),t)=t^{\alpha/(1+p-\alpha(mp-1))}f(\rho;C\pm \epsilon), ~\forall  \rho < \xi_*, ~t\geq 0.
\end{equation}
(Furthermore, we denote the right-hand side of \eqref{a1} by $f(\rho, C)$). Now taking $x=\xi_{\rho}(t)$ in \eqref{E13}, after multiplying by $t^{-\alpha/(1+p-\alpha(mp-1))}$~and passing to the limit, first as $t\rightarrow 0^+$ and then as $\epsilon\rightarrow 0^+$,
we can easily derive \eqref{U1}. Similarly, from \eqref{E13}, \eqref{U14} and \eqref{I5}, \eqref{I1} easily follows. The lemma is proved.
\end{proof}

\begin{proof}[Proof of \cref{lemma 2}] 
As in the proof of \cref{lemma 1}, \eqref{E11} and \eqref{E12} follow from \eqref{IF2}. From \cite{Ishige} and \eqref{CT2} it follows that the existence, uniqueness, and comparison result for the minimal solution of the CP \eqref{OP1}, \eqref{IF1} with $u_0=(C\pm \epsilon)(-x)_+^{\alpha}  \text{ and } T=+\infty$ hold. As before, from \eqref{E11} and \eqref{E12}, \eqref{E13} follows. For arbitrary $k>0$, the function
\begin{equation}\label{U15}
u_k^{\pm \epsilon}(x,t)=ku_{\pm \epsilon}\Big(k^{-1/\alpha}x, k^{(\alpha(mp-1)-(1+p))/\alpha}t\Big), \, k>0,
\end{equation}
is a minimal solution of the following problem:
\begin{subequations}\label{OP2}
\begin{align}
u_t-\Big(|(u^{m})_x|^{p-1}(u^{m})_x\Big)_x+bk^{(\alpha(mp-\beta)-(1+p))/\alpha}u^{\beta}=0,\, x\in \mathbb{R},\, t>0, \label{c1}\\
u(x,0)=(C\pm \epsilon)(-x)_+^{\alpha}, \, x\in \mathbb{R}. \label{IF4}
\end{align}
\end{subequations}
Since $\alpha(mp-\beta)-(1+p)<0,$ it follows that:
\begin{equation}\label{LIM1}
\underset{k\rightarrow+\infty}\lim u_k^{\pm \epsilon}(x,t)=v_{\pm\epsilon}(x,t),~x\in \mathbb{R},~t\geq 0,
\end{equation}
where $v_{\pm \epsilon}$ is a minimal solution to CP \eqref{OP1}, \eqref{IF1} with $b=0,~u_0=(C\pm \epsilon)(-x)_+^{\alpha}, \text{and}~T=+\infty$.~Hence,~$v_{\pm \epsilon}$ satisfies \eqref{U14}. Taking $x=\xi_{\rho}(t),$ where $\rho<\xi_*$ is fixed, from \eqref{LIM1} it follows that for arbitrary $t>0$
\begin{equation}\label{LIM2}
\underset{k\rightarrow+\infty}\lim ku_{\pm \epsilon}\Big(k^{-1/\alpha}\xi_{\rho}(t), k^{(\alpha(mp-1)-(1+p))/\alpha}t\Big)=t^{\alpha(1+p-\alpha(mp-1))}f(\rho;C\pm \epsilon).
\end{equation}
Letting $\tau= k^{(\alpha(mp-1)-(1+p))/\alpha}t,$ then \eqref{LIM2} implies:
\begin{equation}\label{U16}
u_{\pm \epsilon}(\xi_{\rho}(\tau),\tau)\sim \tau^{\alpha/(1+p-\alpha(mp-1))}f(\rho;C\pm \epsilon),~\text{as}~\tau\rightarrow 0^{+}.
\end{equation}
As before, \eqref{U1} easily follows from \eqref{E13} and \eqref{U16}. The lemma is proved.
\end{proof}
\begin{proof}[Proof of \cref{lemma tw}] 
Plugging $u(x,t)=\phi((kt-x))$ into \eqref{OP1} and choosing $y_0 = 0$ we have the following intial value problem for $\phi$:
\begin{equation}\label{TWIVP1}\begin{cases}
\big(|(\phi^{m})'|^{p-1}(\phi^{m})'\big)'- k\phi'-b\phi^{\beta} = 0, ~0\leq y < +\infty \\
\phi(0) = (\phi^m)'(0) = 0,
\end{cases}\end{equation}
Proof of the existence and uniqueness of the solution to \eqref{TWIVP1}, which is monotonically increasing with asymptotic formula \eqref{LIM1TW} is known in particular cases $p=1$ \cite{HerreroVazquez1} and $m=1$ \cite{Mu1}. The standard proof based on phase plane analysis applies with minor modifications. By introducing new variables: 
\begin{equation}\label{phaseplane}
X = \phi, \ Y = ((\phi^m)')^{p},
\end{equation} 
we have the following problem ODE problem in phase plane:
\begin{equation}\label{TWIVP2}\begin{cases}
\frac{dY}{dX} = k + bmX^{m+\beta-1}Y^{-\frac{1}{p}}, \\
Y(0) = 0.
\end{cases}\end{equation}
Since $m+\beta>1$, similar proof as in \cite{Mu1} implies the existence and uniqueness of the global increasing solution of \eqref{TWIVP2}.  Next, we employ a scaling argument to prove: 
\begin{equation}\label{TWS1}
Y(X) \sim \Bigg[\frac{bm(1+p)}{p(m+\beta)}\Bigg]^{\frac{p}{1+p}}X^{\frac{p(m+\beta)}{1+p}}, \text { as } X \rightarrow +\infty.
\end{equation}
Rescaled function:
\begin{equation}
Y_l(X) = lY(l^{\gamma}X), \ \gamma = -\frac{1+p}{p(m+\beta)}, \ l>0,
\end{equation}
solves the problem:
\begin{equation}\label{Scale2}
\frac{dY_l}{dX} = kl^{\frac{p(m+\beta)-(1+p)}{p(m+\beta)}}+bmX^{m+\beta-1}Y_l^{-\frac{1}{p}}, \ Y_l(0)=0.
\end{equation}
As in \eqref{TWIVP2}, there exists a unique global solution $Y_l$ of \eqref{Scale2}. It can be easily shown that the sequences $\{Y_l\}$ and $\{dY_l/dX\}$ are bounded in every fixed compact subset $\mathbb{R}^+$ uniformly for $l\in (0,1]$. By choosing the expanding sequence of compact subsets of $\mathbb{R}^+$, and by applying Arzela-Ascoli theorem and Cantor's diagonalization, it follows that there is a sub-sequence $\{Y_{l'}\}$ which converges as $l'\downarrow 0$ in $\mathbb{R}^+$, and the convergence is uniform on compact subsets of $\mathbb{R}^+$. Since the limit function is a unique solution of the problem \eqref{Scale2} with $l=0$, we have       
\begin{equation}\label{Scale3}
\lim_{l \downarrow 0} Y_l(X) =  \Bigg[\frac{bm(1+p)}{p(m+\beta)}\Bigg]^{\frac{p}{1+p}}X^{\frac{p(m+\beta)}{1+p}}.
\end{equation}
By changing the variable $Z = l^{\gamma}X$, from \eqref{Scale3}, \eqref{TWS1} follows.  

Let $Y$ be a solution of the problem \eqref{TWIVP2}. Note that the problem:
\begin{equation}\label{TW2}
\frac{d\phi^m}{dy}=Y^{\frac{1}{p}}(\phi(y)), \ \phi(0)=0,
\end{equation}
has a unique maximal solution in $(-\infty, M)$, such that $\phi \equiv 0$ for $y\leq 0$, and $\phi>0$ for $0<y<M$. Moreover, whether $M <$ or $= +\infty$, we have $\phi(y) \to +\infty$ as $y\to M-$. From \eqref{TW2} it follows that:
\begin{equation}\label{TWINT1}
m \displaystyle \int_{0}^{\phi(y)}X^{m-1}Y^{-\frac{1}{p}}(X)dX = y, 0<y<M.
\end{equation}
Passing to the limit as $y\to M^-$, from \eqref{TWINT1}, \eqref{TWS1} it follows that $M=+\infty$, and accordingly $\phi$ is a unique global solution of \eqref{TW2}. Equivalently, this implies that $\phi$ is a solution of \eqref{TWIVP1}. By integrating \eqref{TW2}, and by using \eqref{TWS1}, asymptotic formula \eqref{LIM1TW} easily follows.
\end{proof}

\begin{proof}[Proof of \cref{lemma 3}] 
Let $u$ be a unique minimal solution of the problem  \eqref{OP1}, \eqref{IF3} \cite{Ishige,Tsutsumi}. Rescaled function:
\begin{equation}\label{NU11}
u_k(x,t)=ku(k^{-1/\alpha}x,k^{\beta - 1}t),~ k>0,
\end{equation}
satisfies \eqref{OP1}, \eqref{IF3}, and therefore:
\begin{equation}\label{NU12}
u(x,t)\leq ku(k^{-1/\alpha}x,k^{\beta -1}t),~k>0.
\end{equation}
As in the proof of \cref{lemma 1}, it follows that \eqref{NU12} is true with equality sign.
If we choose $k=t^{1/(1-\beta)},$ then \eqref{NU12} implies \eqref{U3} with $f_1(\zeta)=u(\zeta,1)$, where $f_1$ is a nonnegative and continuous solution of \eqref{selfsimilarNODE3}, \eqref{selfsimilarNODE3bc}. By \cite{Barenblatt1}, PDE
\eqref{OP1} has a finite speed of propagation property, and minimal solution of \eqref{OP1}, \eqref{IF3} has a finite interface. Therefore, upper bound $\xi_*$ of the support of $f$ is finite;  $f$ is positive and smooth for $\xi<\xi_*$ and $f=0$ for $\xi\geq \xi_*.$ Now we prove that if $C>C_*$, then $f_1(0)=A_1>0$. We divide the proof into two cases:

\textbf{Case 1}: $p(m+\beta) < 1+ p$ \\
It is enough to show that $\exists$ $t_0 > 0$ such that $u(0,t_0) > 0$. Let $g(x,t) = C_1(t-x)^\frac{1+p}{mp-\beta}_+$, $C_1 \in (C_*,C)$. 
\begin{equation}\nonumber\\
Lg = b C_1^\beta (t-x)^\frac{\beta (1+p)}{mp - \beta}_+ \left [ 1 - \left( \frac{C_1}{C_*} \right)^{mp - \beta} + C_1^{1-\beta} \frac{1+p}{b(mp - \beta)} (t-x)^\frac{1+p-p(m+\beta)}{mp - \beta}_+\right ]
\end{equation}
Since $C_1<C$, we can choose $x_1<0$ and $\delta>0$ such that:
\[ Lg \leq 0, \, \text{in} \ Q:=\{(x,t): x_1\leq x <t, 0<t\leq \delta\}, \]
\[ g(x,0)\leq u(x,0),~x_1\leq x; \ g(x_1,t)\leq u(x_1,t),~0\leq t \leq \delta. \]
Comparison \eqref{CT} implies:
\[ 0 < g(x,t) \leq u(x,t), \forall~ x_1 \leq x < t,~ 0 \leq t \leq \delta. \]
In particular, we have: $u(0,t_0) > 0, \forall~ 0 < t_0 \leq \delta$, which implies that $f_1(0) = A_1 >0$.

\textbf{Case 2}: $p(m+\beta) > 1+ p$ \\
We apply \cref{lemma tw} with the forward traveling wave ($k > 0$). By \eqref{LIM1TW} for some $M > 0$ we have:
\begin{equation}\label{forwardwave}
\phi(y) < Cy^{\frac{1+p}{mp-\beta}}, \text{ for } y > M.
\end{equation}
Let us choose:
\begin{equation}\label{forwardwave1}
K = \max\{\phi(y): 0 \leq y \leq M\}, \ \xi = \max\bigg\{M;\Big(\frac{K}{C}\Big)^{\frac{mp-\beta}{1+p}}\bigg\},
\end{equation}
and consider a family of traveling-wave solutions to \eqref{OP1} of the form: $g(x,t) = \phi((kt-x-\xi))$.  From \eqref{forwardwave},\eqref{forwardwave1} it follows that:
\begin{equation}\label{forwardwave3}
\phi((-x-\xi)) \leq C(-x)^{\frac{1+p}{mp-\beta}}_+, \text{ for any } x \in \mathbb{R}.
\end{equation}
From the comparison theorem it follows that $g \leq u$ for any $x \in \mathbb{R}, ~t \geq 0$. By choosing $t_0 > 0$ such that $k > \xi t^{-1}_0 > 0$, we ensure that:
\begin{equation}\label{A1positive}
0 < g(0, t_0) = \phi((kt_0 - \xi)) \leq u(0, t_0) = t_0^{\frac{1}{1-\beta}}f_1(0),
\end{equation} 
which proves that $f_1(0) > 0$.
To prove the asymptotic formula \eqref{U10} we proceed as we did in the proof in \cref{lemma 2}. As before, \eqref{E11}-  \eqref{E13} follow from \eqref{IF2}, where $v_{\pm \epsilon}$ is a solution of the problem:
\begin{align}
v_t-\Big(|(v^{m})_x|^{p-1}(v^{m})_x\Big)_x+bv^{\beta}=0,~ |x|<|x_{\epsilon}|,~0<t\leq \delta, \\
v(x_{\epsilon},t)=(C\pm \epsilon)(-x_{\epsilon})^{\alpha},~ v(-x_{\epsilon},t)=u(-x_{\epsilon},t),  0\leq t\leq \delta, \\
v(x,0)=(C\pm \epsilon)(-x)_+^{\alpha},~ |x|\leq |x_{\epsilon}|. 
\end{align}
Rescaled function:
\[u_k^{\pm \epsilon}(x,t)=ku_{\pm \epsilon}\Big(k^{-\frac{1}{\alpha}}x,k^{\beta-1}t\Big),~k>0,\]
satisfies the Dirichlet problem:
\begin{subequations}
\begin{align}
v_t=\Big(|(v^{m})_x|^{p-1}(v^{m})_x\Big)_x-bv^{\beta}, ~\text{in}~ E^k_{\epsilon}=\Big\{|x|<k^{\frac{1}{\alpha}}|x_{\epsilon}|, ~0<t\leq k^{1-\beta}\delta\Big\}, \\
v(k^{\frac{1}{\alpha}}x_{\epsilon},t)=k(C\pm \epsilon)(-x_{\epsilon})^{\alpha},~v(-k^{\frac{1}{\alpha}}x_{\epsilon},t)=ku(-x_{\epsilon},k^{\beta-1}t),~0\leq t\leq k^{1-\beta}\delta, \\
v(x,0)=(C\pm \epsilon)(-x)_+^{\alpha},~|x|\leq k^{\frac{1}{\alpha}}|x_{\epsilon}|.
\end{align}
\end{subequations}
As before in the proof of \cref{lemma 2} we have:
\begin{equation}\label{NLIM3}
\underset{k\rightarrow+\infty}\lim u_{k}^{\pm \epsilon}(x,t)=v_{\pm \epsilon}(x,t),~(x,t)\in P := \big\{(x,t):x\in \mathbb{R},~ 0<t\leq t_0\big\},
\end{equation}
thus,
\begin{equation}\label{NU2}
v_{\pm \epsilon}(x,t) = t^\frac{1}{1-\beta}f_1(\rho; C \pm \epsilon), \forall \rho < \zeta_*, t \geq 0
\end{equation}
Taking $x = \eta_\rho (t) = \rho t^\frac{mp-\beta}{(1+p)(1-\beta)}$ and $\tau =k^{\beta-1}t$ it follows from \eqref{NLIM3} that:
\begin{equation}\label{NU17}
u_{\pm \epsilon}(\eta_{\rho}(\tau),\tau)\sim \tau^\frac{1}{1-\beta}f_1(\rho; C\pm \epsilon),~\text{as}~\tau\rightarrow 0^{+}.
\end{equation}
From \eqref{E13} and \eqref{NU17}, since $\epsilon>0$ is arbitrary and $f_1(0) = A_1 > 0$, the desired asymptotic formula \eqref{U10} follows. The lemma is proved. 
\end{proof}


\begin{proof}[Proof of \cref{lemma 4}] 
As before, \eqref{E11}-\eqref{E13} follow from \eqref{IF2}, where $v_{\pm \epsilon}$ is a solution of the problem:
\begin{align}
v_t-\Big(|(v^{m})_x|^{p-1}(v^{m})_x\Big)_x+bv^{\beta}=0,~ |x|<|x_{\epsilon}|,~0<t\leq \delta, \\
v(x_{\epsilon},t)=(C\pm \epsilon)(-x_{\epsilon})^{\alpha},~ v(-x_{\epsilon},t)=u(-x_{\epsilon},t),  0\leq t\leq \delta, \\
v(x,0)=(C\pm \epsilon)(-x)_+^{\alpha},~ |x|\leq |x_{\epsilon}|. 
\end{align}
Rescaled function:
\[u_k^{\pm \epsilon}(x,t)=ku_{\pm \epsilon}\Big(k^{-\frac{1}{\alpha}}x,k^{\beta-1}t\Big),~k>0,\]
satisfies the Dirichlet problem:
\begin{subequations}
\begin{align}
v_t=k^\frac{1+p-\alpha(mp-\beta)}{\alpha}\Big(|(v^{m})_x|^{p-1}(v^{m})_x\Big)_x-bu^{\beta}, ~\text{in}~ E^k_{\epsilon}, \\
v(k^{\frac{1}{\alpha}}x_{\epsilon},t)=k(C\pm \epsilon)(-x_{\epsilon})^{\alpha},~v(-k^{\frac{1}{\alpha}}x_{\epsilon},t)=ku(-x_{\epsilon},k^{\beta-1}t),~0\leq t\leq k^{1-\beta}\delta, \\
v(x,0)=(C\pm \epsilon)(-x)_+^{\alpha},~|x|\leq k^{\frac{1}{\alpha}}|x_{\epsilon}|, 
\end{align}
\end{subequations}
where \[E^k_{\epsilon}:=\Big\{|x|<k^{\frac{1}{\alpha}}|x_{\epsilon}|, ~0<t\leq k^{1-\beta}\delta\Big\}.\]
The next step consists in proving the convergence of the sequence $\{u_k^{\pm \epsilon}\}$ as $k\rightarrow +\infty$. 
This step is identical with the proof given in the similar Lemma 3.4 from \cite{Abdulla1}. For any fixed $t_0>0$, the function $g(x,t)=(C+1)(1+x^2)^{\alpha/2}e^{t}$ is a uniform upper bound for the sequence $\{u_k^{\pm \epsilon}\}$ in
$E_{0\epsilon}^{k}=E_{\epsilon}^k\cap P$, where$P= \{(x,t): 0<t\leq t_0\}.$ The sequences $\{u_k^{\pm \epsilon}\}$ are uniformly H\"older continuous on an arbitrary compact subset of $P$ \cite{Est-Vazquez,Ivanov1}. As in the proof of the Lemma 3.4 of \cite{Abdulla1} it is proved that some subsequences $\{u_{k'}^{\pm \epsilon}\}$ converge to solutions of the reaction equation. This imply that
\begin{equation}\label{U17}
u_{\pm \epsilon}(\eta_{\ell}(\tau),\tau)\sim \tau^{\frac{1}{1-\beta}}\Big[(C\pm \epsilon)^{1-\beta}\ell^{\alpha(1-\beta)}-b(1-\beta)\Big]^{\frac{1}{1-\beta}},~\text{as}~\tau\rightarrow 0^{+}.
\end{equation}
From \eqref{E13} and \eqref{U17}, since $\epsilon>0$ is arbitrary, the desired formula \eqref{U5} follows. The lemma is proved.
\end{proof}
	\section{Proofs of the Main Results}\label{sec: proofs of the main results.} \begin{proof}[Proof of \cref{theorem 1}] 
From \cref{lemma 2} and \eqref{U1} it follows
\begin{equation}\label{LIM4}
\underset{t\rightarrow 0^+}\lim \text{inf}~\eta(t)t^{1/(\alpha(mp-1)-(1+p))}\geq \xi_*.
\end{equation}
For $\forall~\epsilon>0$, let $u_{\epsilon}$ be a minimal solution of the CP \eqref{OP1}, \eqref{IF3} with $b=0$ and with $C$ replaced by $C+\epsilon$. The second inequality of \eqref{E11} and the first inequality of \eqref{E12} follow from \eqref{IF2}. Since $u_{\epsilon}$ is a supersolution of \eqref{OP1}, from \eqref{E11}, \eqref{E12}, and a comparison principle, the second inequality of \eqref{E13} follows. By \cref{lemma 1} we have:
\[\eta(t)\leq(C+\epsilon)^\frac{mp-1}{1+p-\alpha(mp-1)}\xi'_*t^{1/(1+p-\alpha(mp-1))},~0\leq t\leq\delta,\]
and hence:
\begin{equation}\label{LIM5}
\underset{t\rightarrow0^+}\lim\text{sup}~\eta(t)t^{1/(\alpha(mp-1)-(1+p))}\leq\xi_*.
\end{equation}
From \eqref{LIM4} and \eqref{LIM5}, \eqref{I1} follows. 
\end{proof}

\begin{proof}[Proof of \cref{theorem 2}]
Assume that $u_0$ is defined by \eqref{IF3} and $p(m+\beta)\neq 1+p.$ The self-similar form \eqref{U3} follows from \cref{lemma 3}. Let $C>C_*$.
For a function:
\begin{equation}\label{G1}
g(x,t) =t^{1/(1-\beta)}f_1(\zeta),~\zeta=xt^{-\frac{mp-\beta}{(1+p)(1-\beta)}}.
\end{equation}
we have
\begin{equation}\label{g1}Lg=t^{\frac{\beta}{1-\beta}}\mathcal{L}^0f_1,
\end{equation}
where the operator $\mathcal{L}^0$ is defined by \eqref{selfsimilarNODE3}.
By choosing
\[f_1(\zeta)=C_0(\zeta_0-\zeta)_+^{\gamma_0},~~0<\zeta<+\infty,\]
with $C_0,~\zeta_0>0$ and $\gamma_0=(1+p)/(mp-\beta)$ we have
\begin{equation}\label{OP8}
\mathcal{L}^0f_1=bC_0^{\beta}(\zeta_0-\zeta)_+^{\frac{\beta(1+p)}{mp-\beta}}\Bigg\{1-(C_0/C_*)^{mp-\beta}+ \frac{C_0^{1-\beta}}{b(1-\beta)}\zeta_0(\zeta_0-\zeta)_+^{\frac{1+p-p(m+\beta)}{mp-\beta}}\Bigg\}.
\end{equation}
For an upper estimation we choose $C_0=C_2 ~\text{and}~ \zeta_0=\zeta_2$ (see the appendix, \cref{sec: appendix}). If $p(m+\beta)>1+p$, we have
\[\mathcal{L}^0f_1\geq bC_2^{\beta}(\zeta_2-\zeta)_+^{\frac{\beta(1+p)}{mp-\beta}}\Bigg\{1-(C_2/C_*)^{mp-\beta}+ \frac{C_2^{1-\beta}}{b(1-\beta)}\zeta_2^{\frac{(1+p)(1-\beta)}{mp-\beta}}\Bigg\}=0,~ \text{for} ~0\leq \zeta \leq \zeta_2,\]
while if $p(m+\beta)<1+p$, we have:
\[\mathcal{L}^0f_1\geq bC_2^{\beta}(\zeta_2-\zeta)_+^{\frac{\beta(1+p)}{mp-\beta}}\Big\{1-(C_2/C_*)^{mp-\beta}\Big\}=0,~ \text{for} ~0\leq \zeta \leq \zeta_2.\]
By \eqref{g1} we have
\begin{subequations}\label{E18}
\begin{align}
Lg\geq 0,~ \text{for} ~0<x<\zeta_2 t^{\frac{mp-\beta}{(1+p)(1-\beta)}},~0<t<+\infty, \label{h1}\\
Lg= 0,~ \text{for} ~x>\zeta_2 t^{\frac{mp-\beta}{(1+p)(1-\beta)}},~0<t<+\infty. \label{h2}
\end{align}
\end{subequations}
\eqref{CT} implies that $g$ is a supersolution of \eqref{OP1} in $\{(x,t): x>0, ~t>0\}$. Since
\begin{subequations}\label{IF7}
\begin{align}
g(x,0)=u(x,0)=0, ~ \text{for} ~0\leq x<+\infty, \label{i1}\\
g(0,t)=u(0,t), ~\text{for} ~0\leq x<+\infty, \label{i2}
\end{align}
\end{subequations}
the right-hand side of \eqref{E1} follows.
If $p(m+\beta)<1+p$, to prove the lower estimation we choose $C_0=C_1,~ \zeta_0=\zeta_1, \text{and}~\gamma_0=(1+p)/(mp-\beta).$ From \eqref{OP8} and \eqref{g1} we have
\[\mathcal{L}^0f_1\leq bC_1^{\beta}(\zeta_1-\zeta)^{\frac{\beta(1+p)}{mp-\beta}}\Bigg\{1-(C_1/C_*)^{mp-\beta}+\frac{C_1^{1-\beta}}{b(1-\beta)}\zeta_1^{\frac{(1+p)(1-\beta)}{mp-\beta}}\Bigg\}=0,~\text{for}~0\leq\zeta \leq \zeta_1,\]
\begin{subequations}\label{E19}
\begin{align}
Lg\leq 0,~ \text{for} ~0<x<\zeta_1 t^{\frac{mp-\beta}{(1+p)(1-\beta)}},~0<t<+\infty,\label{j1}\\
Lg= 0,~ \text{for} ~x>\zeta_1 t^{\frac{mp-\beta}{(1+p)(1-\beta)}},~0<t<+\infty. \label{j2}
\end{align}
\end{subequations}
As before from \eqref{IF7} and \eqref{CT}, the left-hand side of \eqref{E1} follows.
If $p(m+\beta)>1+p$, then to prove the lower estimation we choose $C_0=C_1, \zeta_0=\zeta_1$ and $\gamma_0=p/(mp-1).$ We have
\begin{gather} 
\mathcal{L}^0f_1\leq C_1(1-\beta)^{-1}(\zeta_1-\zeta)^{\frac{1+p-mp}{mp-1}}\times\nonumber\\
\times\Bigg\{\zeta_1-C_1^{mp-1}\frac{(1-\beta)p(mp)^{p}}{(mp-1)^{1+p}}+b(1-\beta)C_1^{\beta -1}\zeta_1^{ \frac{p (m+\beta)-(1+p)}{mp-1}}\Bigg\}=0,~0<\zeta<\zeta_1, \nonumber
\end{gather} 
which again implies \eqref{E19}. From \eqref{CT}, the left-hand side of \eqref{E1} follows.

Let $p(m+\beta)>1+p$ and $0<C<C_*.$ For $\gamma\in [0,1)$ consider a function
\[g(x,t)=\bigg[C^{1-\beta}(-x)_+^{\frac{(1+p)(1-\beta)}{mp-\beta}}-b(1-\beta)(1-\gamma)t\bigg]_+^{\frac{1}{1-\beta}},~x\in\mathbb{R},~t>0.\]
We estimate $Lg$ in
\[M:=\{(x,t):-\infty<x<\mu_{\gamma}(t),~t>0\}, \
\mu_{\gamma}(t)=-\bigg[b(1-\beta)(1-\gamma)C^{\beta-1}t\bigg]^{\frac{mp-\beta}{(1+p)(1-\beta)}}.\]
We have $Lg=bg^{\beta}S$, where
\begin{subequations}\label{S1}
\begin{align}
S= \gamma-C^{mp-\beta}\Bigg[1-\Bigg(\frac{-\mu_{\gamma}(t)}{(-x)_+}\Bigg)^{\frac{(1+p)(1-\beta)}{mp-\beta}}\Bigg]^\frac{p(m+\beta-1)-\beta}{1-\beta} \times \nonumber\\
\times \Bigg[R_1 + R_2 \Bigg[1-\Bigg(\frac{-\mu_{\gamma}(t)}{(-x)_+}\Bigg)^{\frac{(1+p)(1-\beta)}{mp-\beta}}\Bigg]^{-1}\Bigg], \label{k1} \\
S|_{t=0} = \gamma-(C/C_*)^{mp-\beta}, ~ S|_{x=\mu_{\gamma}(t)}=\gamma,
\end{align}
\end{subequations}
where $R_1, R_2 >0$ (see \cref{sec: appendix}). Moreover,
\[S_t \geq 0 ~\text{in} ~M.\]
Thus,
\[\gamma-(C/C_*)^{mp-\beta}\leq S\leq\gamma ~ \text{in} ~M.\]
If we take $\gamma=(C/C_*)^{mp-\beta}$ (respectively, $\gamma=0$), then we have:
\begin{subequations}\label{E20}
\begin{align}
Lg\geq 0 ~(\text{respectively}, Lg\leq 0)~\text{in} ~M, \label{l1}\\
Lg= 0, ~\text{for}~x>\mu_{\gamma}(t),~t>0. \label{l2}
\end{align}
\end{subequations}
From \eqref{CT}, the estimation \eqref{E3} follows.
Let $p(m+\beta)<1+p$ and $0<C<C_*.$ First, we establish the following rough estimation:
\begin{align}\label{E21}
\bigg[C^{1-\beta}(-x)_+^{\frac{(1+p)(1-\beta)}{mp-\beta}}-b(1-\beta)\big(1-(C/C_*)^{mp-\beta}\big)t\bigg]_+^{\frac{1}{1-\beta}}\leq  \nonumber\\
\leq u(x,t) \leq C(-x)_+^{\frac{1+p}{mp-\beta}},~ \text{for}~x\in \mathbb{R},~0\leq t<+\infty.
\end{align}
To prove the left-hand side we consider the function, $g$, as in the case when $p(m+\beta)>1+p$~with $\gamma = \big(C/C_*\big)^{mp-\beta}.$ As before, we then derive \eqref{k1}, and since: 
\[S_t\leq 0, \, \text{in} \, M,\]
we have $S\leq 0 ~\text{in}~ M.$ Hence, \eqref{E20} is valid with $\leq$ in \eqref{l1}. As before, from \eqref{CT}, the left-hand side of \eqref{E21} follows. To prove the right-hand side of \eqref{E21} it is enough to observe that:
\[Lu_0=bu_0^{\beta}\big(1-(C/C_*)^{mp-\beta}\big)\geq 0,~\text{for}~x\in \mathbb{R},~t\geq0.\]
Having \eqref{E21}, we can now establish a more accurate estimation \eqref{E4}. Consider a function:
\[g(x,t)=C_0\Big(-\zeta_0 t^{\frac{mp-\beta}{(1+p)(1-\beta)}}-x\Big)^{\frac{1+p}{mp-\beta}}_+,~\text{in}~ G_{\ell},\]
\[G_{\ell}:=\Big\{(x,t): \zeta(t)=-\ell t^{\frac{mp-\beta}{(1+p)(1-\beta)}}<x<+\infty,~0<t<+\infty\Big\},\]
where, $C_0,~\zeta_0>0, \ell>\zeta_0$. Calculating $Lg$ in
\[G^+_\ell:=\Big\{(x,t): \zeta(t)<x<-\zeta_0 t^{\frac{mp-\beta}{(1+p)(1-\beta)}},~0<t<+\infty\Big\},\]
we have:
\begin {align}\label{S3}
Lg=bg^{\beta}S,\quad S=1-(C_0/C_*)^{mp-\beta}-(b(1-\beta))^{-1}C_0^{1-\beta}\zeta_0 t^{\frac{p(m+\beta)-(1+p)}{(1+p)(1-\beta)}}\times\ \nonumber\\
\times \Big(-\zeta_0 t^{\frac{mp-\beta}{(1+p)(1-\beta)}}-x\Big)^{\frac{1+p-p(m+\beta)}{mp-\beta}}.
\end{align}
By choosing $C_0=C_*,$ we have:
\begin{equation}\label{E22}
Lg\leq 0, \text{ in}~ G_{\ell}^+;~Lg=0,~\text{in}~G_{\ell}\backslash \bar G_{\ell}^+.
\end{equation}
To obtain a lower estimation we choose $\zeta_0=\zeta_3, \text{and}~ \ell=\ell_0$ (see \cref{sec: appendix}). Using \eqref{E21}, we have:
\begin{subequations}\label{E23}
\begin{align}
g|_{x=\zeta(t)}=t^{\frac{1}{1-\beta}}C_*(\ell_0-\zeta_3)^{\frac{1+p}{mp-\beta}}=\Big(b(1-\beta)\theta_*t\Big)^{\frac{1}{1-\beta}}
=t^{\frac{1}{1-\beta}}\times\nonumber\\\times\Bigg[C^{1-\beta}\ell_0^{\frac{(1+p)(1-\beta)}{mp-\beta}}-b(1-\beta)\Big(1-\big(C/C_*\big)^{mp-\beta}\Big)\Bigg]^{\frac{1}{1-\beta}}\leq u(\zeta(t),t),~t\geq0,\label{m1}\\
g(x,0)=u(x,0)=0, ~0\leq x\leq x_0, \label{m2}\\
g(x_0,t)=u(x_0,t)=0, ~t\geq 0, \label{m3}
\end{align}
\end{subequations}
where $x_0>0$ is an arbitrary fixed number. By using \eqref{E22}, \eqref{E23}, we can apply \eqref{CT} in:
\[G'_{\ell_0}:=G_{\ell_0}\cap \Big\{x<x_0\Big\}.\]
Since $x_0>0$ is arbitrary, the left inequality in \eqref{E4} follows. Since
\[S_x\geq 0, ~\text{for}~ \zeta(t)<x<-\zeta_0 t^{\frac{mp-\beta}{1+p)(1-\beta)}},~t>0,\]
from \eqref{S3} it follows that
\[S\geq S|_{x=\zeta((t)}=1-(C_0/C_*)^{mp-\beta}-\big(b(1-\beta)\big)^{-1}C_0^{1-\beta}\zeta_0 (\ell-\zeta_0 )^{\frac{1+p-p(m+\beta)}{mp-\beta}}.\]
By choosing $ C_0=C_3,~\zeta_0=\zeta_4,~\ell=\ell_1$ (see \cref{sec: appendix}), we have
\[S|_{x=\zeta(t)}=0,\]
\[Lg\geq 0~\text{in}~G^+_{\ell_1},~Lg=0~\text{in}~G_{\ell_1}\backslash \bar G^+_{\ell_1},\]
\[u(\zeta(t),t)\leq t^{\frac{1}{1-\beta}} C\ell_1^{\frac{1+p}{mp-\beta}}=t^{\frac{1}{1-\beta}}C_3(\ell_1-\zeta_4)^{\frac{1+p}{mp-\beta}}=g(\zeta(t),t),~t\geq 0,\]
and, for arbitrary $x_0>0,$ \eqref{m2} and \eqref{m3} are valid. By applying \eqref{CT} in $G'_{\ell_1},$  due to the arbitrariness of $x_0>0$, we derive the right-hand side of \eqref{E4}.
From \eqref{E1}, \eqref{E3}, and \eqref{E4} it follows that:
\[\zeta_1t^{\frac{mp-\beta}{(1+p)(1-\beta)}}\leq \eta(t)\leq \zeta_2t^{\frac{mp-\beta}{(1+p)(1-\beta)}},~0\leq t<+\infty,\]
where the constants $\zeta_1$ and $\zeta_2$ are chosen according to relevant estimations for $u$. From \eqref{I2} and the respective estimations \eqref{E1}, \eqref{E3}, and \eqref{E4}, the estimation \eqref{E2} follows. If $u_0$ satisfies \eqref{IF2} with $\alpha=(1+p)/(mp-\beta)$ and with $C\neq C_*$, then the asymptotic formulae \eqref{I3} and \eqref{U4} may be proved as the similar estimations \eqref{I1} and \eqref{U1} were in \cref{lemma 1}. 
\end{proof}
\begin{proof}[Proof of \cref{theorem 3}] 
For $\forall~\epsilon>0$ from \eqref{IF2}, \eqref{E11} follows. Consider a function:
\begin{equation}\label{G2}
g_{\epsilon}(x,t)=\Big[(C+\epsilon)^{1-\beta}(-x)_+^{\alpha(1-\beta)}-b(1-\beta)(1-\epsilon)t\Big]_+^{1/(1-\beta)}.
\end{equation}
We estimate $Lg$ in:
\[M_1:=\big\{(x,t):x_{\epsilon}<x<\eta_{\ell}(t), \, 0<t \leq \delta_1\big\},\]
\[\eta_{\ell}(t)=-\ell t^{1/(\alpha(1-\beta)}, ~\ell(\epsilon)=(C+\epsilon)^{-1/\alpha}\big[b(1-\beta)(1-\epsilon)\big]^{1/\alpha(1-\beta)},\]
where  $\delta_1>0$ is chosen such that $\eta_{\ell(\epsilon)}(\delta_1)=x_{\epsilon}.$ We have
\[Lg_{\epsilon}=bg_{\epsilon}^{\beta}\{\epsilon+S\}\]
\[S=-b^{-1}(\alpha m)^{p}(C+\epsilon)^{mp-\beta}(-x)_+^{\alpha(mp-\beta)-(1+p)}\Big\{g|x|^{-\alpha}\Big/(C+\epsilon)\Big\}^{p( m+\beta)-(1+p)}S_1,\]
\[S_1=\Bigg\{\alpha p(m+\beta-1) +p(\alpha(1-\beta)-1)\Big[g|x|^{-\alpha}\Big/ (C+\epsilon)\Big]^{1-\beta}\Bigg\}.\]
If $ p(m+\beta) \geq 1+p$, we can choose $x_{\epsilon}<0$ such that
\[|S|<\frac{\epsilon}{2},~\text{in}~M_{1}.\]
Thus we have:
\begin{align}
Lg_{\epsilon}>b(\epsilon/2)g^{\beta}_{\epsilon}~ (\text{respectively}, ~ Lg_{-\epsilon}<-b\big(\epsilon/2)g^{\beta}_{-\epsilon}\big),~\text{in}~M_1, \\
Lg_{\pm \epsilon}=0, ~\text{for}~x>\eta_{\ell(\pm \epsilon)}(t),~0<t\leq \delta_1, \\
g_{\epsilon}(x,0)~\geq~ u_0(x)~\big(\text{respectively}, ~g_{-\epsilon}(x,0)\leq u_0(x)\big),~x\geq x_{\epsilon}. \
\end{align}
Since $u$ and $g$ are continuous functions, $\delta=\delta(\epsilon)\in (0,\delta_1]$, may be chosen such that:
\[g_{\epsilon}(x_{\epsilon},t)\geq u(x_{\epsilon},t)~\big(\text{respectively},~ g_{-\epsilon}(x_{\epsilon},t)\leq u(x_{\epsilon},t)\big),~0\leq t\leq \delta.~\]
From \eqref{CT} it follows that:
\begin{subequations}\label{E24}
\begin{align}
g_{-\epsilon}\leq u\leq g_{\epsilon}, \, x\geq x_{\epsilon}, \, 0\leq t\leq \delta, \label{n1}\\
\eta_{\ell(-\epsilon)}(t)\leq \eta(t)\leq \eta_{\ell(\epsilon)},~ 0\leq t\leq \delta, \label{n2}
\end{align}
\end{subequations}
which imply \eqref{I4} and \eqref{U5}.
Let $p(m+\beta)<1+p$. In this case the left-hand side of \eqref{E24} may be proved similarly. Moreover, we can replace $1+\epsilon$ with $1$ in $g_{-\epsilon}$ and $\eta_{\ell}(-\epsilon).$
For a sharp upper estimation, consider a function:
\[g(x,t)=C_6\Big(-\zeta_5t^{\frac{1}{\alpha(1-\beta)}}-x\Big)_+^{\alpha},~\text{in}~G_{\ell,\delta},\]
\[G_{\ell,\delta}:=\{(x,t):\eta_{\ell}(t)<x<+\infty,~0<t<\delta\},\]
where~$\ell \in (\ell_*,+\infty)$,
$C_6$ and $\zeta_5$ are defined in \cref{sec: appendix}.
From \eqref{U5} it follows that for all $\ell>\ell_*$ and for all sufficiently small $\epsilon>0$, there exists a $\delta=\delta(\epsilon,\ell)>0$ such that:
\begin{equation}\label{E26}
u(\eta_{\ell}(t),t)\leq t^{\frac{1}{1-\beta}}[C^{1-\beta}\ell^{\alpha(1-\beta)}-b(1-\beta)(1-\epsilon)]^{\frac{1}{1-\beta}},~0\leq t\leq\delta.
\end{equation}
Calculating  $Lg$ in
\[G^+_{\ell,\delta}=\Big\{(x,t):\eta_{\ell}(t)<x<-\zeta_5t^{\frac{1}{\alpha(1-\beta)}},~0<t<\delta\Big\},\]
we derive
\[Lg=bg^{\beta}S,~S=1-(b(1-\beta))^{-1}\zeta_5 C_6^{1/\alpha}\Big\{gt^{1/(\beta-1)}\Big\}^{(\alpha(1-\beta)-1)/\alpha}\]
\[-\frac{p}{b}(\alpha m)^{p}(\alpha m -1)C_6^{(1+p)/\alpha}g^{mp-\beta-((1+p)/\alpha)}.\]
Since
\[S_x\geq 0,~\text{in}~G^+_{l,\delta},\]
\[S \geq S\vert_{x=\eta_{\ell}(t)} =1-(b(1-\beta))^{-1}\zeta_5 C_6^{1-\beta}(\ell-\zeta_5)^{\alpha(1-\beta)-1}-\]
\[-t^\frac{\alpha(mp-\beta)-(1+p)}{\alpha(1-\beta)}b^{-1}p(\alpha m)^{p}(\alpha m-1)C_6^{mp-\beta}(\ell-\zeta_5)^{\alpha(mp-\beta)-(1+p)},\]
we have
\[ S\geq\epsilon-t^\frac{\alpha(mp-\beta)-(1+p)}{\alpha(1-\beta)}b^{-1}p(\alpha m)^{p}(\alpha m-1)C_6^{mp-\beta}(\ell-\zeta_5)^{\alpha(mp-\beta)-(1+p)}, \ \quad{in} \ G^+_{\ell,\delta}.\]
By choosing $\delta = \delta(\epsilon)>0$ sufficiently small we have
\begin{equation}\label{o1}
Lg \geq b(\epsilon/2)g^{\beta}, \text{ in } G^+_{\ell, \delta}.
\end{equation}
By applying \eqref{E26} and \eqref{CT} in $G'_{\ell,\delta}=G_{\ell,\delta}\cap \{x<x_0\}$ we have
\begin{subequations}\label{E27}
\begin{align}
Lg=0, ~\text{in}~G'_{\ell,\delta} \backslash \bar G^+_{\ell,\delta}, \label{o2}\\
u(\eta_{\ell}(t),t)\leq t^{\frac{1}{1-\beta}}\big[C^{1-\beta}\ell^{\alpha(1-\beta)}-b(1-\beta)(1-\epsilon)\big]^{\frac{1}{1-\beta}} = \nonumber \\
= C_6(\ell-\zeta_5)^{\alpha}t^{\frac{1}{\alpha(1-\beta)}}=g(\eta_{\ell}(t),t),~0\leq t\leq \delta, \label{o3}\\
u(x_0,t)=g(x_0,t)=0,~0\leq t\leq \delta, \label{o4}\\
u(x,0)~=g(x,0)~=0,~0\leq x\leq x_0. \label{o5}
\end{align}
\end{subequations}
Since $x_0>0$ is arbitrary, from \eqref{E27} and \eqref{CT}, it follows that for all $\ell>\ell_*$ and $\epsilon>0$, there exists $\delta=\delta(\epsilon,\ell)>0$ such that:
\begin{equation}\label{E29}
u(x,t)\leq C_6\Big(-\zeta_5 t^{\frac{1}{\alpha(1-\beta)}}-x\Big)_+^{\alpha},~\text{in}~\bar G_{\ell,\delta.}
\end{equation}
In view of \eqref{U5} (which is valid along $x=\eta_{\ell}(t)$),  $\delta$ may be chosen so small that:
\begin{equation}\label{E30}
-\ell t^{1/\alpha(1-\beta)}\leq \eta(t)\leq -\zeta_5t^{1/\alpha(1-\beta)},~0\leq t\leq\delta.
\end{equation}
Since $\ell>\ell_*$ and $\epsilon>0$ are arbitrary numbers, \eqref{I4} follows from \eqref{E30}.
\end{proof}
Proof of \cref{theorem 4}, and all the results described in \cref{sec: details of the main results} case (4a)-(4d), and in the special case of $b=0$ are almost identical to the similar proofs given in \cite{Abdulla1}.
	
	\section{Numerical Solution}\label{sec: numerical solution}

In this section, we investigate the numerical solutions to \eqref{OP1} using on a weighted essentially nonoscillatory (WENO) scheme. In \cref{sec:weno}, we briefly introduce the WENO discretization for the PDE. Numerical results and comparisons with analysis are presented in \cref{sec:ZKB} and \cref{sec:compare}. All figures can be found in the appendix, \cref{sec:fig}.   

\subsection{Finite Difference Discretization} 
\label{sec:weno}

WENO methods refer to a family of finite volume or finite difference methods for solutions of hyperbolic conservation laws and other convection dominated problems. The central idea behind the WENO scheme is to use nonlinear combinations of numerical stencils for solution interpolation/reconstruction, with weights adapted to the smoothness of the solution on these stencils. Therefore, interpolation across discontinuous or nonsmooth part of the solution is avoided as much as possible. This yields numerical solutions with high order accuracy in smooth regions, while maintaining non-oscillatory and sharp discontinuity transitions \cite{WENO_REVIEW}. These features make WENO schemes well suitable for the study of problems with piecewise smooth solutions containing discontinuities or sharp interfaces. A WENO scheme was proposed in \cite{WENO_LIU} to solve nonlinear degenerate parabolic equation of the form $u_t = (b(u))_{xx}$. In that paper, the second order derivative term is directly approximated using a conservative flux difference formula. Below we describe a finite difference WENO scheme for the solution of the nonlinear double degenerate parabolic equation \eqref{OP1}.  

As shown in \cref{fig:weno}, the numerical solution is defined at full grid node $u_{i} = u(x_i)$, where $x_i = i\Delta x$ and $\Delta x = x_{i+1}-x_i$ is the uniform grid size. Defining the flux function $f(u^m _x) = |(u^{m})_x|^{p-1}(u^{m})_x$, and introduce an auxillary function $h(\xi)$ such that:
\begin{equation}
f(u^m _x)(x) = \frac{1}{\Delta x}\int _{x-\frac{\Delta x}{2}} ^{x+\frac{\Delta x}{2}}h(\xi)d\xi .
\label{eq:avg}
\end{equation}      
Then at grid node $x_i$, we have: 
\begin{equation}
\Big(|(u^{m})_x|^{p-1}(u^{m})_x\Big)_x(x_i) = f(u^m _x)_x(x_i) = \frac{1}{\Delta x}(h(x_{i+\frac{1}{2}})-h(x_{i-\frac{1}{2}})).
\label{eq:hx}
\end{equation}
Notice that to evaluate the derivative $f(u^m _x)_x$ at $x_i$, we need the values of $h(x)$ at half grid nodes $x_{i-\frac{1}{2}} = x_i-\Delta x$ and $x_{i+\frac{1}{2}} = x_i+\Delta x$. Therefore, if the function $h(x)$ can be computed to $r$th order of accuracy, then the right hand side of equation \eqref{eq:hx} would be an $r$th order approximation to $f(u^m _x)$. Overall, the WENO approximation for $\Big(|(u^{m})_x|^{p-1}(u^{m})_x\Big)_x$ can be summarized as following:
\begin{enumerate}
\item With the given values of $u(x)$ (and thus the values of $u^m(x)$) defined at grid $x_i=i\Delta x$, approximate the derivative $u^m _x$ at $x_i$ using a fifth order WENO interpolation scheme. Based on this, compute the pointwise values of the flux function $f(u^m _x)(x_i)$. From \eqref{eq:avg}, this value is also the cell average of $h(x)$ over the interval $(x_{i-\frac{1}{2}},x_{i+\frac{1}{2}})$. 
\item From the cell average of $h(x)$, compute the values of $h(x)$ at the half grid nodes $x_{i+\frac{1}{2}}$ with a fifth order WENO reconstruction scheme.    
\end{enumerate}
Here the WENO interpolation scheme for $u^m _x$ is similar to that proposed in \cite{WENO_JIANG} for the solution of Hamilton-Jacobi Equations. In \cite{WENO_JIANG}, derivatives of the solution are computed on left-biased and right-biased numerical stencils to construct monotone Hamiltonians. For the solution of the nonlinear parabolic equation in this paper, we simply calculate the left and right biased derivatives by WENO approximation and take the arithmetic average of the two to be the value of $u^m _x$. A similar strategy is used for the construction of the value $h(x)$ at half grid nodes, where a fifth order WENO reconstruction scheme \cite{WENO_REVIEW} is applied.

To compute the numerical solution at a new time level $t_{k+1}=t_k+\Delta t$ from its value at time level $t_k$, we apply the third-order TVD Runge-Kutta time discretization \cite{WENO_REVIEW}:
\begin{equation}
u^{(1)} = u^{k}+\Delta t L(u^{k}), \\ \\
u^{(2)} = \frac{3}{4}u^{k}+\frac{1}{4}u^{(1)}+\frac{1}{4}\Delta t L(u^{(1)}), \\ \\
u^{(k+1)} = \frac{1}{3}u^{k}+\frac{2}{3}u^{(2)}+\frac{2}{3}\Delta t L(u^{(2)}).
\end{equation}
Here, $\Delta t$ is the time step, $u^{k}$ and $u^{k+1}$ are numerical solutions at time level $t_k$ and $t_{k+1}$, respectively. And $L(u)$ is the finite difference approximation to the right hand side of equation.

\subsection{Comparison with the Instantaneous Point Source (IPS) Solution}
\label{sec:ZKB}
Solution of the CP for \eqref{OP1}, $b=0$ with initial function being a Dirac's point mass (or $\delta$-function) is given by (\cite{zeldovich,Barenblatt1}):
\begin{equation}\label{ZKB:u}
u_*(x,t) = t^\frac{-1}{p(m+1)} \left[\Gamma - k(m,p) \left(|x|t^{\frac{-1}{p(m+1)}}\right)^{\frac{1+p}{p}}\right]_+^{\frac{p}{mp -1}},
\end{equation}
with $k(m,p) =  \frac{mp-1}{m(1+p)}\left(\frac{1}{p(m+1)}\right)^\frac{1}{p}$. Here $\Gamma$ is an integration constant, defined by the conservation of the energy. This solution has a compact support $[-\eta(t),\eta(t)]$ with the interface function given by:
\begin{equation}\label{ZKB:eta}
\eta(t) = t^{\frac{1}{p(m+1)}}\left(\frac{\Gamma}{k(m,p)}\right)^{\frac{p}{p+1}}.
\end{equation}
For our numerical test, we use parameters $m=6$, and $p=2$ and $3$ respectively. The computational domain is $[-5,5]$ with total number of $1024$ grid points. The initial condition is taken as the IPS solution $u_*$ at $t = 0.05$ with $\Gamma=1.0$. We set the domain large enough so that the interface does not reach the boundary at the end of numerical simulation. The periodic boundary condition is imposed to simplify the numerical implementation. Without the focus on efficiency of the algorithm, we always choose the time step $\Delta t$ small enough to get a stable solution.     

The comparison between the numerical and analytical solutions is shown in \cref{fig:zkbu} for time $t=2.0$. Here the filled circles represent the numerical solution and the solid line is the solution \eqref{ZKB:u}. The agreement is excellent. The WENO scheme can successfully capture the sharp transition in the solution without generating any apparent numerical oscillations. To identify the location of the (right) interface $\eta(t)$, we take the first $x_i$ where $u_i < 10^{-10}$ as the interface location. In \cref{fig:zkbeta}, the computed values for $\eta(t)$ (circles) are plotted together with that given by the IPS solution \eqref{ZKB:eta} (solid curve) at different stages of the simulation. It is clear that the dynamics of the interface is accurately captured by the WENO scheme. 
%
%

\subsection{Comparison with Analytical Results}
\label{sec:compare}

In this section, we apply the WENO scheme to equation \eqref{OP1} with initial condition given by \eqref{IF3}. To compare the numerical solution with the analytical results for the CP, we use numerical initial condition as shown in \cref{fig:IC}. Here $u(x,0)$ is given by condition \eqref{IF3} near the interface (interval $[-1,0]$ for this case). As the value of $x$ gets smaller, $u(x,0)$ is smoothly brought to zero by a hyperbolic tangent function. Notice that since the solution to \eqref{OP1} has a finite speed of propagation, it is expected that as long as the time is short enough, the numerical solution locally close to the interface should agree with that from the analysis. For all the numerical examples shown below, a grid size of $\Delta x = \frac{1}{128}$ is used. Since the interfaces never reach beyond the domain boundary at the end of the simulation, periodic boundary conditions are applied.
%
%
\subsubsection{Region 1 with Expanding Interface} 

For region 1, we choose $m=4$, $p=2$, $\beta = 0.5$, $b=0.5$, $C=1.0$, and $\alpha = 0.2 < (1+p)/(mp-\beta) = 0.4$. For these parameters, the interface expands and its location is given by $\eta(t) \sim \xi _{*}t^{1/(1+p-\alpha(mp-1))} = \xi _{*}t^{0.625}$ for a positive $\xi _{*}$. In order to compare numerical results with analysis, we need to solve the second order nonlinear ODE \eqref{selfsimilarODE1}, up to $\xi = \xi _{*}$ where $f(\xi _{*}) = 0$. Since $\xi _{*}$ is unknown, we transfer the BVP \eqref{selfsimilarODE1}, \eqref{selfsimilarODE1bc} to a system of IVP, by introducing another variable $g(\xi) = (f^m(\xi))' = mf^{m-1}(\xi)f'(\xi)$. We then solve the system with some given initial conditions. However, since the boundary condition \eqref{selfsimilarODE1bc} is given at negative infinity in the analysis, it is not clear how to set the initial conditions for $f(\xi)$ and $g(\xi)$, respectively.   

From the analysis, we know that $f(\xi)t^{\frac{\alpha}{1+p-\alpha(mp-1)}} \sim u(x,t)$ as $t \to 0^+$, along the curve $x = \xi t^{1/(1+p-\alpha(mp-1))}$. Thus one strategy is to use the numerical solution near the interface to estimate $f(\xi)$ and its derivative at specific value of $\zeta$. Therefore, we have the approximation $f(0) \approx u(0,t)t^{-\frac{\alpha}{1+p-\alpha(mp-1)}} = u(0,t)/t^{0.125}$ for small time $t$. Specifically, we choose to approximate $f(0)$ by 
\begin{equation}\label{eq:f0}
f(0) \approx \frac{1}{3} \sum_{i=1}^3 u(0,t_i)/t_i ^{0.125}, 
\end{equation}
where $t_1=0.01$, $t_2=0.02$ and $t_3=0.03$, respectively. Plug in the values of the numerical solution at $x=0$ and $t=t_i$, we get $f(0) \approx 0.752$. To evaluate $f'(0)$, we use the approximation $f(\xi _0) \approx u(\Delta x,t)/t^{0.125}$ and $f(2\xi _0) \approx u(2\Delta x,t)/t^{0.125}$ for small time $t$, where $\xi _0 = \Delta x/t^{1/(1+p-\alpha(mp-1))} = \Delta x/t^{0.625}$. The evaluation of $f(\xi _0)$ and $f(2\xi _0)$ is similar to \eqref{eq:f0} for $f(0)$. Then we fit a quadratic function to interpolate the three points $(0,f(0))$, $(\xi _0,f(\xi _0))$ and $(2\xi _0,f(2\xi _0))$ and use the derivative of the qudratic function at $0$ to approximate $f'(0)$. Through some simple calculation, we get $f'(0) \approx -0.309$. Finally, the values of $f(0)$ and $f'(0)$ are used as initial conditions in the ODE solver (third-order TVD Runge-Kutta Discretization) to solve for $g(\xi)$ and $f(\xi)$ . The numerical solutions for $f(\xi)$ and its derivative are plotted in \cref{fig:xi} (a) and (b), respectively. As the value of $\xi$ increases, function $f(\xi)$ decreases and the rate of decreasing gets larger. As $\xi \sim \xi _{*}$, the function becomes nonsmooth and the ODE solver fails to yield an accurate solution, even with a very small time step. We choose $\xi _{*}$ to be the value of $\xi$ which gives the smallest $f(\xi)$ and get $\xi _{*} \approx 0.696$. In \cref{fig:exp}(a), we plot the interface location computed by the WENO scheme together with the analytical curve $\eta(t) = \xi _{*}t^{0.625}$. Good agreement is achieved for small time intervals. 

We can estimate the range for $\xi _{*}$ based on analytical results given by \eqref{XI1}, \eqref{XI2} and bound $[\xi_1, \xi_2]$ given in \cref{sec: appendix}. In addition to the constants $C$, $m$, $p$ and $\alpha$, the solution to the CP \eqref{OP1}, \eqref{IF3} $w(0,1)$ with $b = 0$ and $C = 1$ is needed. We approximate the value $w(0,1)$ using the numerical solution from WENO scheme and get $w(0,1) \approx 0.725$. Then from \eqref{XI1}, \eqref{XI2} and bound $[\xi_1, \xi_2]$ given in \cref{sec: appendix} we compute the bounds for the range of $\xi _{*}$ as $\xi _1 = 0.678$ and $\xi _2 = 0.764$. It is clear that the value of $\xi _{*}$ computed before is within the range. In \cref{fig:exp}(b), it is shown that without the absorption term, the interface location given by the numerical solution is indeed bounded by the two curves predicted by analysis.

\subsubsection{Region 2 the Borderline Case}

For Region 2, we first choose $m=2.5$, $p=0.5$, $\beta=0.5$, $b=1.0$, $\alpha = (1+p)/(mp-\beta) = 2.0$. Thus $p(m+\beta) = 1+p = 1.5$. With these parameters, we have $C_{*} \approx 0.13572$. With the choice of $C > C_{*}$ and $C < C_{*}$, the analytical solution is given by the explicit formula \eqref{U2}. In \cref{fig:borderline1}, the numerical results show excellent agreement with the analytical traveling wave solution.

For the second set of tests, we choose $m=2.0$, $p=2.0$, $\beta=0.2$, $b=1.0$, and $\alpha = (1+p)/(mp-\beta) \approx 0.78947$. With these choices, we have $p(m+\beta) > 1+p$ and $C_{*} \approx 0.75655$. For this case, the analytical results are given by \eqref{U3} and \eqref{I2}. Here we use the similar strategy as described in the previous section to numerically estimate $\zeta _{*}$ through the solution of the nonlinear ODE \eqref{selfsimilarNODE3}. For the choice of $C = 1.2 > C_{*}$ and $C = 0.2 < C_{*}$, we get the estimation $\zeta _{*} \approx 12.3$ and $\zeta _{*} \approx -6.7$, respectively. The comparison between numerical solution and analysis is shown in \cref{fig:borderline2}. In the plot, the analytical curves are given by $\eta(t) = 12.3t^{1.25}$ for \cref{fig:borderline2}(a) and $\eta(t) = -6.7t^{1.25}$ for \cref{fig:borderline2}(b). The numerical results agree well with the analysis for short time durations. 

Finally we choose $m=0.5$, $p=2.0$, $\beta=0.2$, $b=1.0$, and $\alpha = (1+p)/(mp-\beta) = 3.75$. With these choices, we have $p(m+\beta) < 1+p$ and $C_{*} \approx 0.1032$. For the choice of $C = 0.4 > C_{*}$ and $C = 0.05 < C_{*}$, we solve the nonlinear ODE \eqref{selfsimilarNODE3} and get the estimation $\zeta _{*} \approx 0.895$ and $\zeta _{*} \approx -0.586$, respectively. As shown in \cref{fig:borderline3}, the agreement between numerics and analysis is again very good. In the plot, the analytical curves are given by $\eta(t) = 0.895t^{\frac{1}{3}}$ for \cref{fig:borderline3}(a) and $\eta(t) = -0.586t^{\frac{1}{3}}$ for \cref{fig:borderline3}(b).  
 \subsubsection{Region 3 with Shrinking Interface}

In Region 3, we choose the parameters $m=4$,$p=2$,~$\beta = 0.5$, $b=0.8$, $C=0.5$, and $\alpha=0.8 > (1+p)/(mp-\beta) = 0.4$. For this choice, the absorption term dominates and the interface shrinks. The analytical solution and interface location are given by \eqref{U5} and \eqref{I4}, respectively. Comparison between numerical and analytical results is plotted in \cref{fig:shrink}. 
It is interesting to note that for the interface location as shown in \cref{fig:shrink}(a), excellent agreement is obtained between the numerics and analysis during the whole simulation, even though the analysis is valid only for short time period. In \cref{fig:shrink}(b), the numerical solution $u(x,t)$ near the interface matches well with that from the analysis.     
\subsubsection{Region 4 with Waiting Time}

In Region 4, we choose $m=2$, $p=3$, $\beta = 1.0$, $b=0.5$, $C=0.5$. Corresponding to the analysis for Region (4a), we set $\alpha = (1+p)/(mp-1) = 0.8$. With these parameters, numerical solutions at different time are plotted in \cref{fig:case4a}(a). It is clear that the interface at $x=0$ remains stationary up to $t=1$. In \cref{fig:case4a}(b), the numerical solution near the interface agrees well with the analytical result given by \eqref{U6}.

\section*{Acknowledgement}
This research was funded by National Science Foundation: grant \#1359074--REU Site: Partial Differential Equations and Dynamical Systems at Florida Institute of Technology (Principal Investigator Professor Ugur G. Abdulla). 
	\FloatBarrier
           \section{Appendix A}\label{sec: appendix}
Here we give explicit values of the constants used in \cref{sec: description of the main results}.\\\\
$\xi_1=p^{\frac{1}{1+p}}\Big(\alpha(mp-1)\Big)^{-\frac{1}{1+p}},~ \xi_2=1, ~\text{if}~ p(mp-1)^{-1}\leq \alpha<(1+p)(mp-1)^{-1}$, 

$\xi_1=1,~\xi_2=p^{\frac{1}{1+p}}\Big(\alpha(mp-1)\Big)^{-\frac{1}{1+p}}, ~\text{if}~ 0<\alpha\leq p(mp-1)^{-1}$;

$\zeta_1=A^{\frac{mp-1}{1+p}}\big(1+b(1-\beta)A_1^{\beta-1}\big)^{-\frac{1}{1+p}}(p(mp)^{p}(1-\beta))^\frac{1}{1+p}(mp-1)^{-1},~~~~$

$C_1=A_1 \zeta_1^{-\frac{p}{mp-1}},~\text{if}~p(m+\beta)>1+p, C>C_*$,\\\\
$\zeta_1=A_1^{\frac{mp-1}{1+p}}\big(1+b(1-\beta)A_1^{\beta-1}\big)^{-\frac{1}{1+p}}((m(1+p))^{p}p(m+\beta)(1-\beta))^{\frac{1}{1+p}}(mp-\beta)^{-1},~~~$\\

$C_1=A_1 \zeta_1^{-\frac{1+p}{mp-\beta}},~\text{if}~p(m+\beta)<1+p, C>C_*$,\\\\
$\zeta_2=A_1^{\frac{mp-1}{1+p}}\big(1+b(1-\beta)A_1^{\beta-1}\big)^{-\frac{1}{1+p}}((m(1+p))^{p}p(m+\beta)(1-\beta))^{\frac{1}{1+p}}(mp-\beta)^{-1},~~~~$\\

$ C_2=A_1 \zeta_2^{-\frac{1+p}{mp-\beta}},~\text{if}~p(m+\beta)>1+p, C>C_*$,\\\\
$\zeta_2=(A_1/C_*)^{\frac{mp-\beta}{1+p}},~C_2=C_* ,~\text{if}~p(m+\beta)<1+p, C>C_*$;\\\\
$\zeta_1=-C^{-\frac{mp-\beta}{1+p}}\big(b(1-\beta)\big)^{\frac{mp-\beta}{(1+p)(1-\beta)}},~\text{if}~p(m+\beta)>1+p, 0<C<C_*$,\\\\
$\zeta_2=-C^{-\frac{mp-\beta}{1+p}}\bigg(b(1-\beta)\big(1-\big(C/C_*)^{mp-\beta}\big)^{\frac{mp-\beta}{(1+p)(1-\beta)}}\bigg),~\text{if}~p(m+\beta)<1+p, 0<C<C_*$,\\\\
$R_1 = (m(1+p))^{p}p(1+p-p(m+\beta))(b(mp-\beta)^{1+p})^{-1},$ \\\\
$R_2 = (m(1+p))^{p}(1+p)p(m+\beta-1)(b(mp-\beta)^{1+p})^{-1},$\\\\
$\theta_*=\Bigg[1-\Big(C/C_*\Big)^{mp-\beta}\Bigg]\Bigg[\Big(C_*/C\Big)^{\frac{(mp-\beta)(1-\beta)}{1+p-p(m+\beta)}}-1\Bigg]^{-1},$\\\\
$\ell_0=C_*^{-\frac{mp-\beta}{1+p}}(C_*/C)^{\frac{(mp-\beta)(1-\beta)}{1+p-p(m+\beta)}}(b(1-\beta)\theta_*)^{\frac{mp-\beta}{(1+p)(1-\beta)}}$,\\\\
$\zeta_3=C_*^{-\frac{mp-\beta}{1+p}}\Big[(C_*/C)^{\frac{(mp-\beta)(1-\beta)}{1+p-p(m+\beta)}}-1\Big](b(1-\beta)\theta_*)^{\frac{mp-\beta}{(1+p)(1-\beta)}}$,\\\\
$\ell_1=C^{-\frac{mp-\beta}{1+p}}\Big[b(1-\beta)(\delta_* \Gamma)^{-1}\Big(1-\delta_* \Gamma-\big(1-\delta_* \Gamma\big)^{-p}(C/C_*)^{mp-\beta}\Big) \Big]^{\frac{mp-\beta}{(1+p)(1-\beta)}}$,\\\\
$\zeta_4=\delta_* \Gamma \ell_1,~\Gamma=1- (C/C_*)^{\frac{mp-\beta}{1+p}},~C_3=C \big(1-\delta_* \Gamma\big)^{-\frac{1+p}{mp-\beta}},$~\\\\
where $\delta_*\in(0,1)$ satisfies:

$g(\delta_*)=\underset{[0,1]}\max g(\delta),\quad ~g(\delta)=\delta^{\frac{1+p-p(m+\beta)}{mp-\beta}}\Big(1-\delta \Gamma-\big(1-\delta \Gamma\big)^{-p}(C/C_*)^{mp-\beta}\Big)$.\\\\
$\bar C=\bigg[\frac{(mp-1)^{1+p}}{p(m+1)(m(1+p))^{p}}\bigg]^\frac{1}{mp-1}$, \ 
$\gamma_{\epsilon} = \frac{p(m+1)(m(1+p))^{p}(C+\epsilon)^{mp-1}}{(mp-1)^{p}} +\epsilon$.\\\\
$\xi_3=A_0^{\frac{mp-1}{1+p}}\Bigg[\frac{(mp)^{p}(1+p-\alpha(mp-1))}{(mp-1)^{p}}\Bigg]^{\frac{1}{1+p}}C^{\frac{mp-1}{1+p-\alpha(mp-1)}}\xi_1$,\\\\
$\xi_4=A_0^{\frac{mp-1}{1+p}}\Bigg[\frac{(mp)^{p}(1+p-\alpha(mp-1))}{(mp-1)^{p}}\Bigg]^{\frac{1}{1+p}}C^{\frac{mp-1}{1+p-\alpha(mp-1)}}\xi_2$,\\\\
$C_4 =C^{(1+p)/(1+p-\alpha(mp-1))}A_0\xi_3^{p/(1-mp)}, \, C_5 =C^{(1+p)/(1+p-\alpha(mp-1))}A_0\xi_4^{p/(1-mp)}$.\\\\
$\zeta_5= (\ell_*/\ell)^{\alpha(1-\beta)}(1-\epsilon)\ell$,\\\\
$C_6=\big[1-(\ell_*/\ell)^{\alpha(1-\beta)}(1-\epsilon)\big]^{-\alpha}\big[C^{1-\beta}-\ell^{-\alpha(1-\beta)}b(1-\beta)(1-\epsilon))\big]^{1/(1-\beta)}.$\\\\

\section{Appendix B}
\label{sec:fig}
Here we list the figures corresponding to the numerical results as described in  \cref{sec: numerical solution}.
\begin{figure}[tbhp]
\centering
\includegraphics[width=0.75\textwidth]{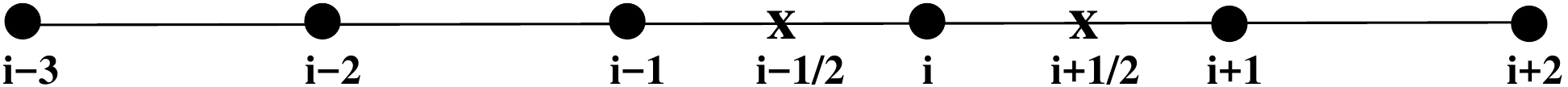}
\caption{Computational grid for the WENO scheme}
\label{fig:weno}
\end{figure}
\begin{figure}[tbhp]
\centering
\subfloat[$p=2$]{\includegraphics[width=0.43\textwidth]{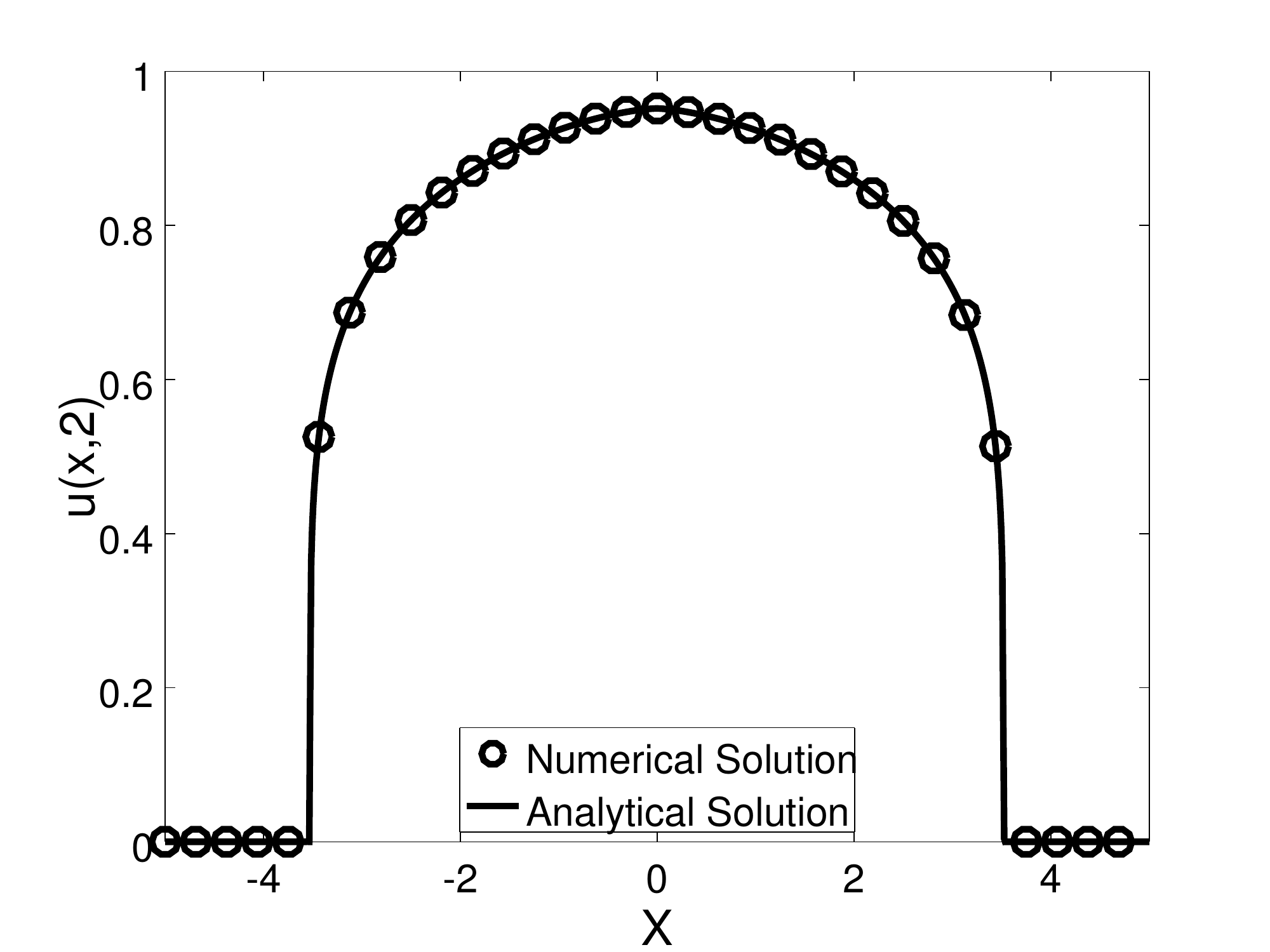}}
\subfloat[$p=3$]{\includegraphics[width=0.43\textwidth]{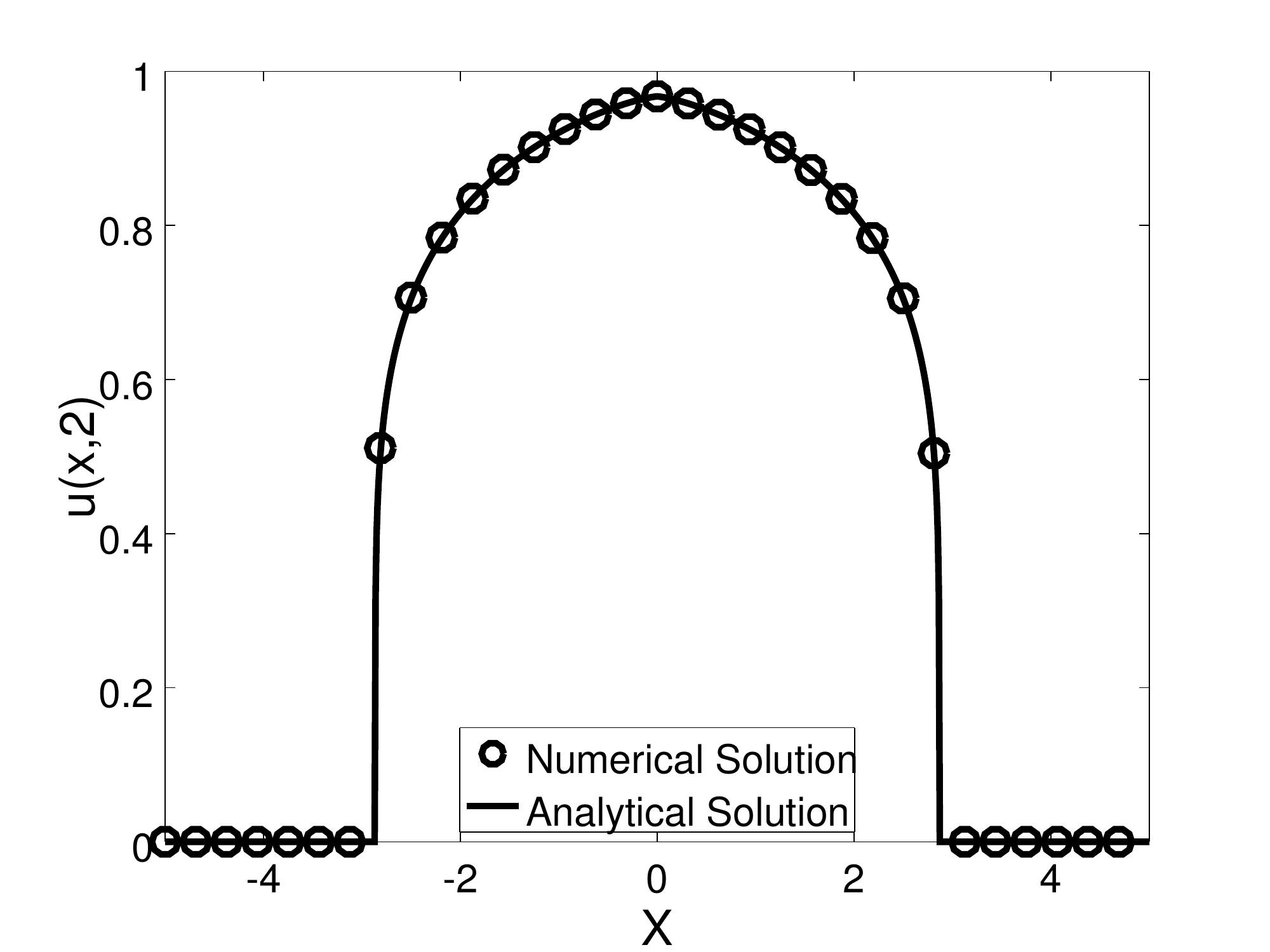}} 
\caption{IPS Solution: numerical and analytical solution at t = 2.0}
\label{fig:zkbu}
\end{figure}
\begin{figure}[tbhp]
\centering
\subfloat[$p=2$]{\includegraphics[width=0.43\textwidth]{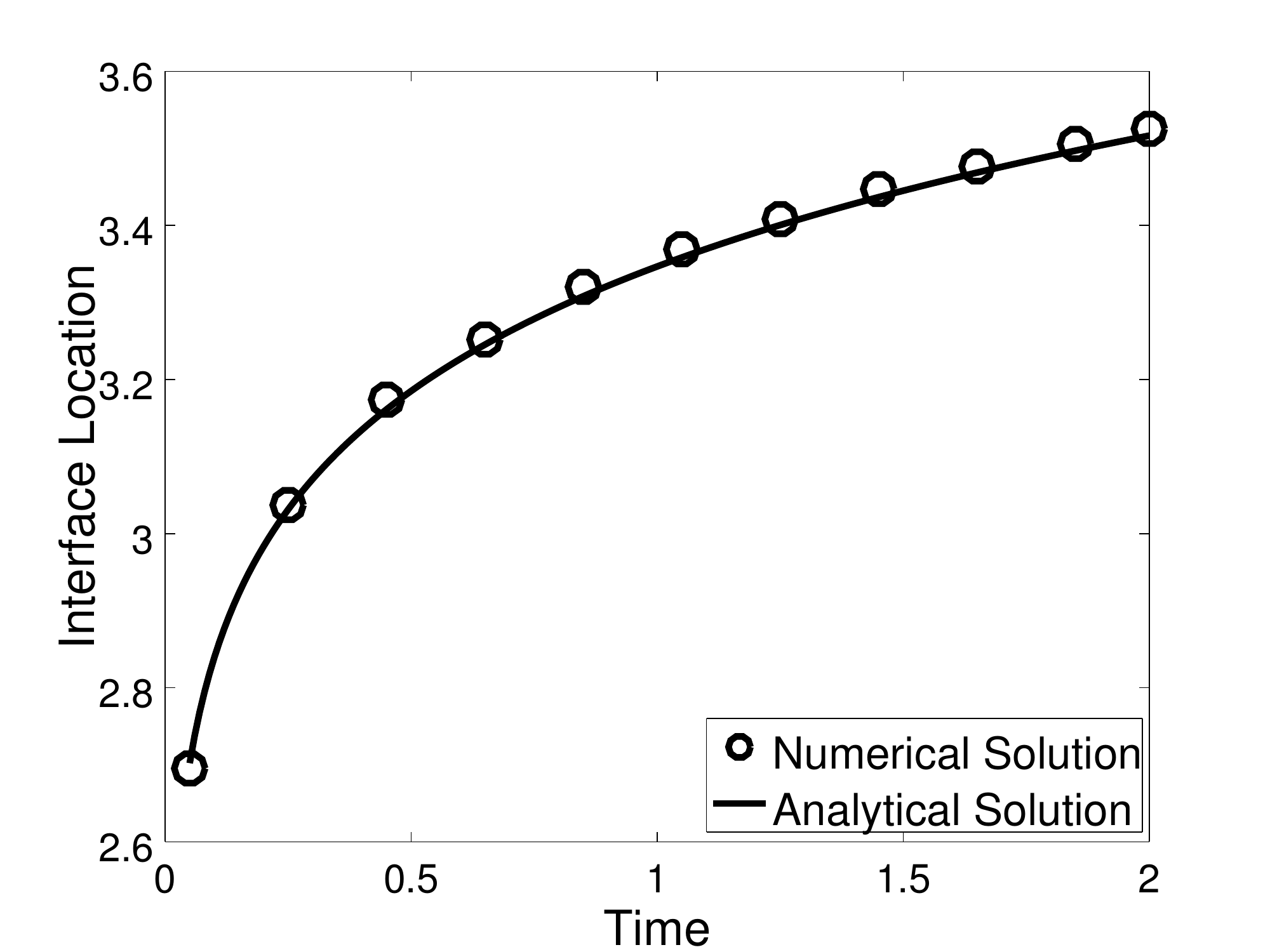}}
\subfloat[$p=3$]{\includegraphics[width=0.43\textwidth]{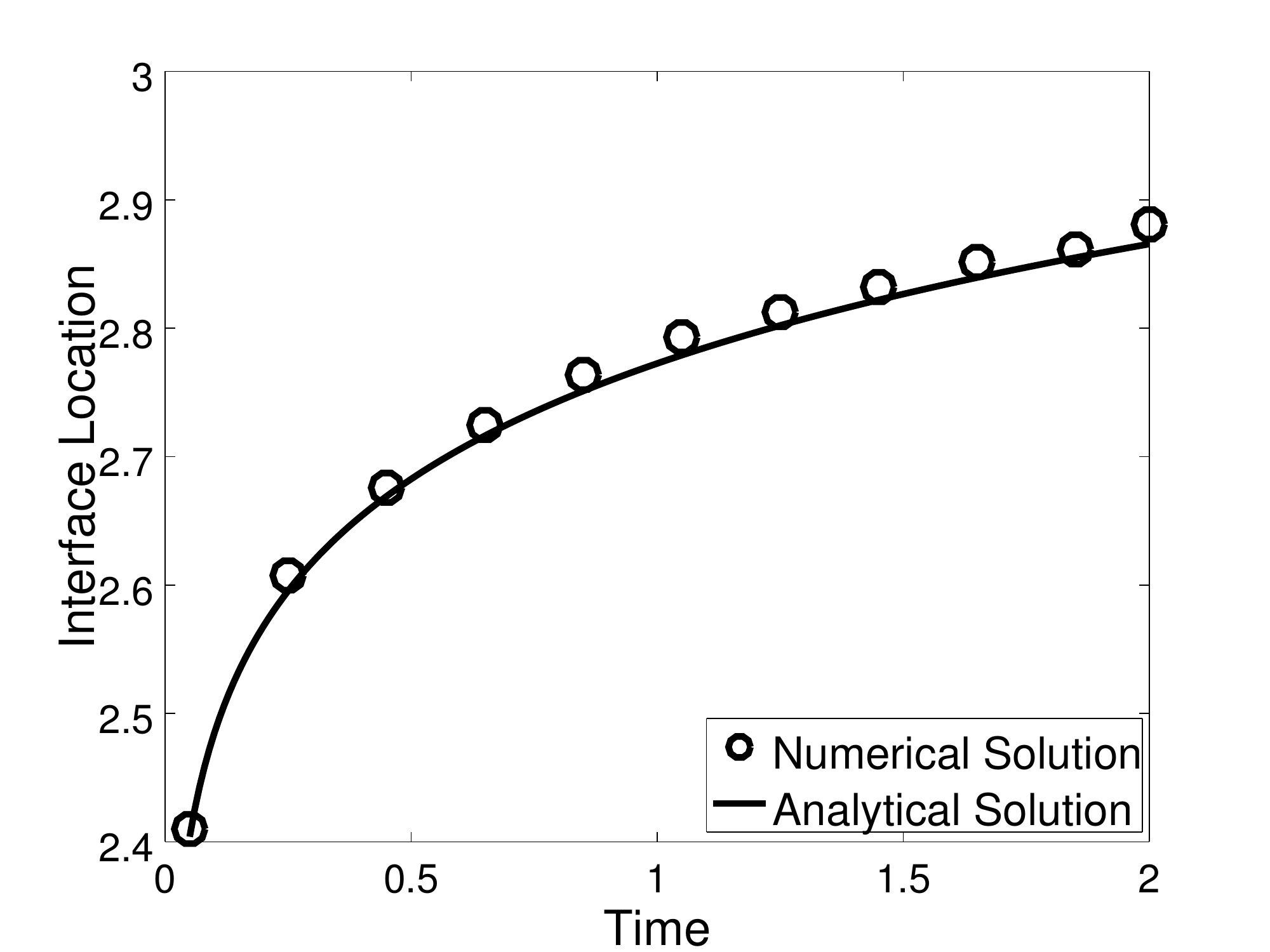}} 
\caption{IPS Solution: Interface Location Vs. Time}
\label{fig:zkbeta}
\end{figure}
\clearpage
\begin{figure}[tbhp]
\centering
\includegraphics[width=0.45\textwidth]{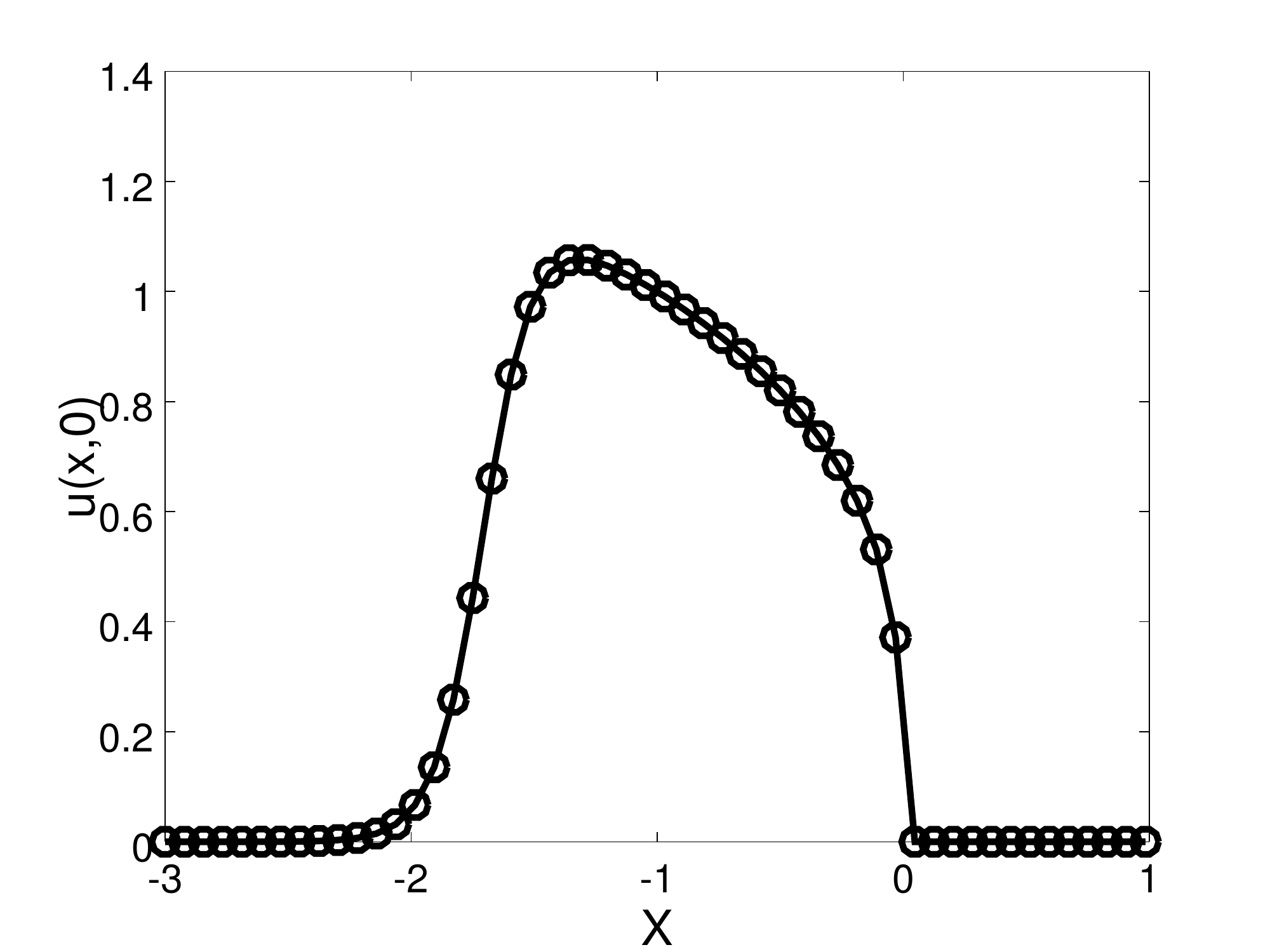}
\caption{Initial condition used for computational study}
\label{fig:IC}
\end{figure}
\begin{figure}[tbhp]
\centering
\subfloat[value of $f(\zeta)$]{\includegraphics[width=0.43\textwidth]{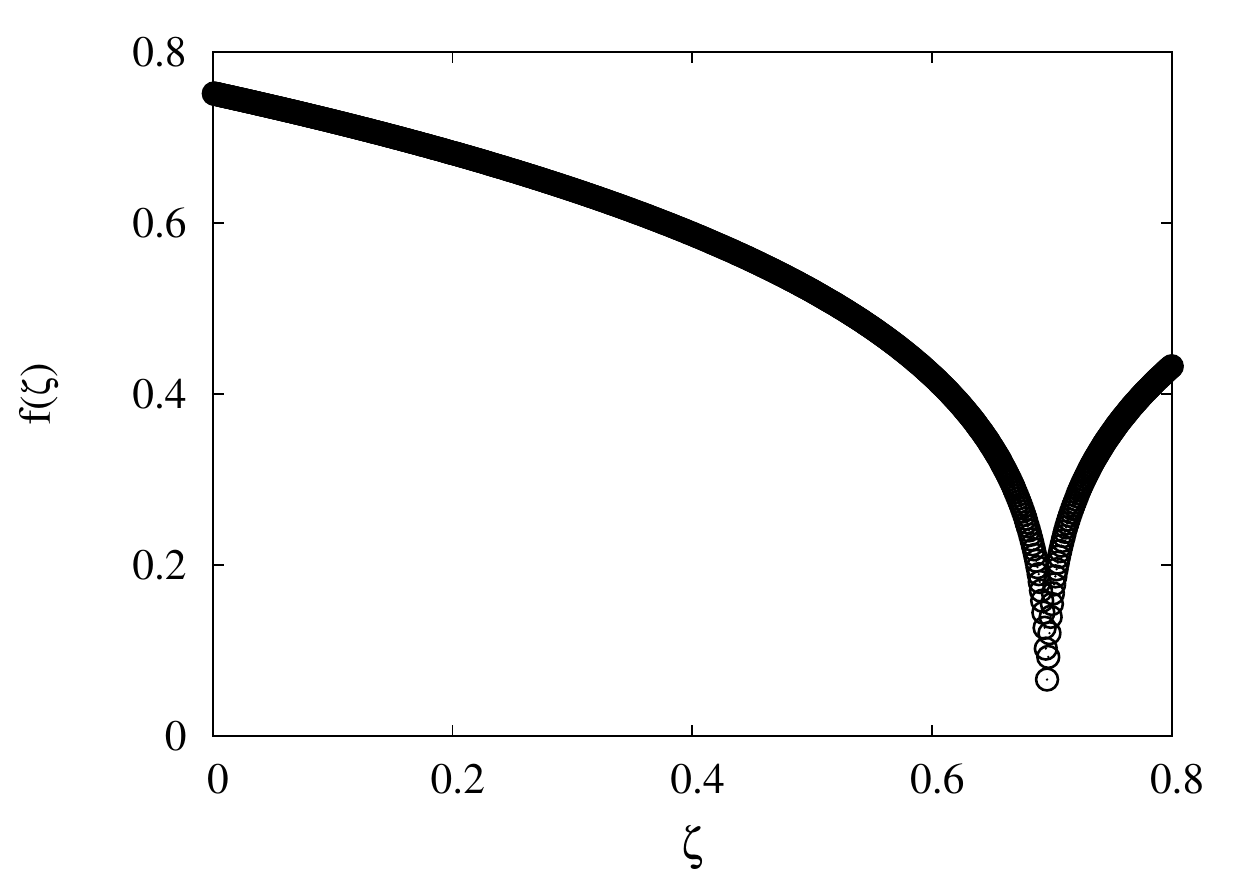}}
\subfloat[derivative of $f(\zeta)$]{\includegraphics[width=0.43\textwidth]{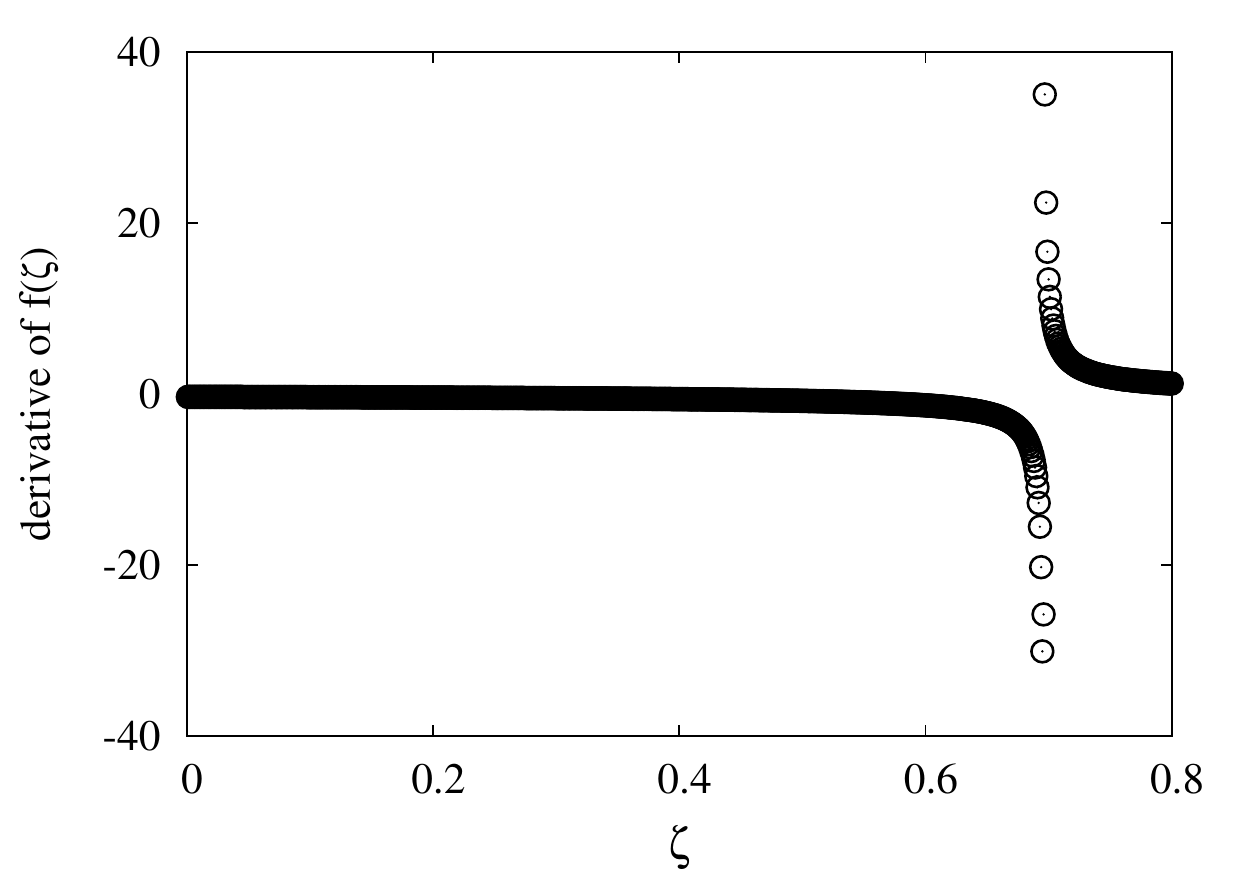}}
\caption{Numerical solution for $f(\zeta)$ and $f'(\zeta)$}
\label{fig:xi}
\end{figure}
\begin{figure}[tbhp]
\centering
\subfloat[$b=0.5$]{\includegraphics[width=0.43\textwidth]{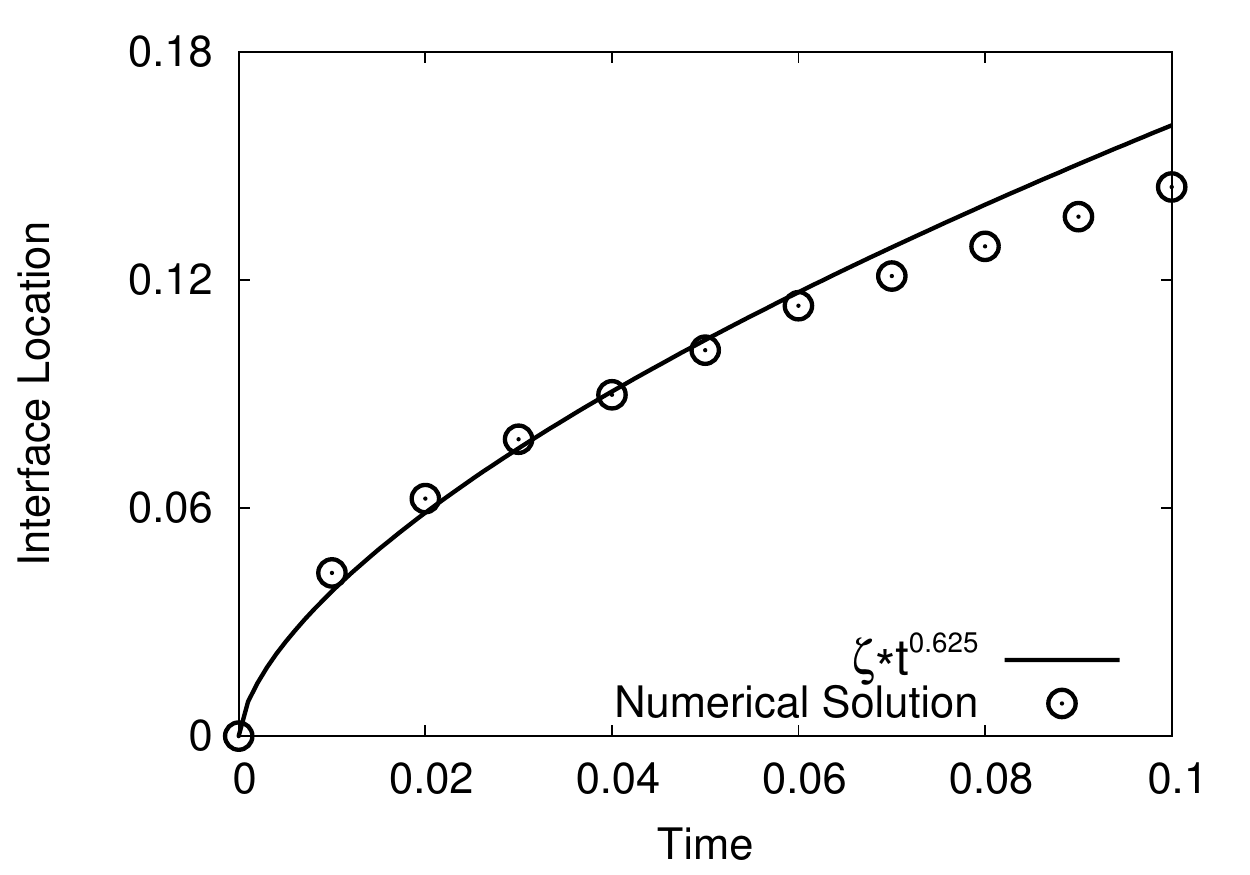}}
\subfloat[$b=0.0$]{\includegraphics[width=0.43\textwidth]{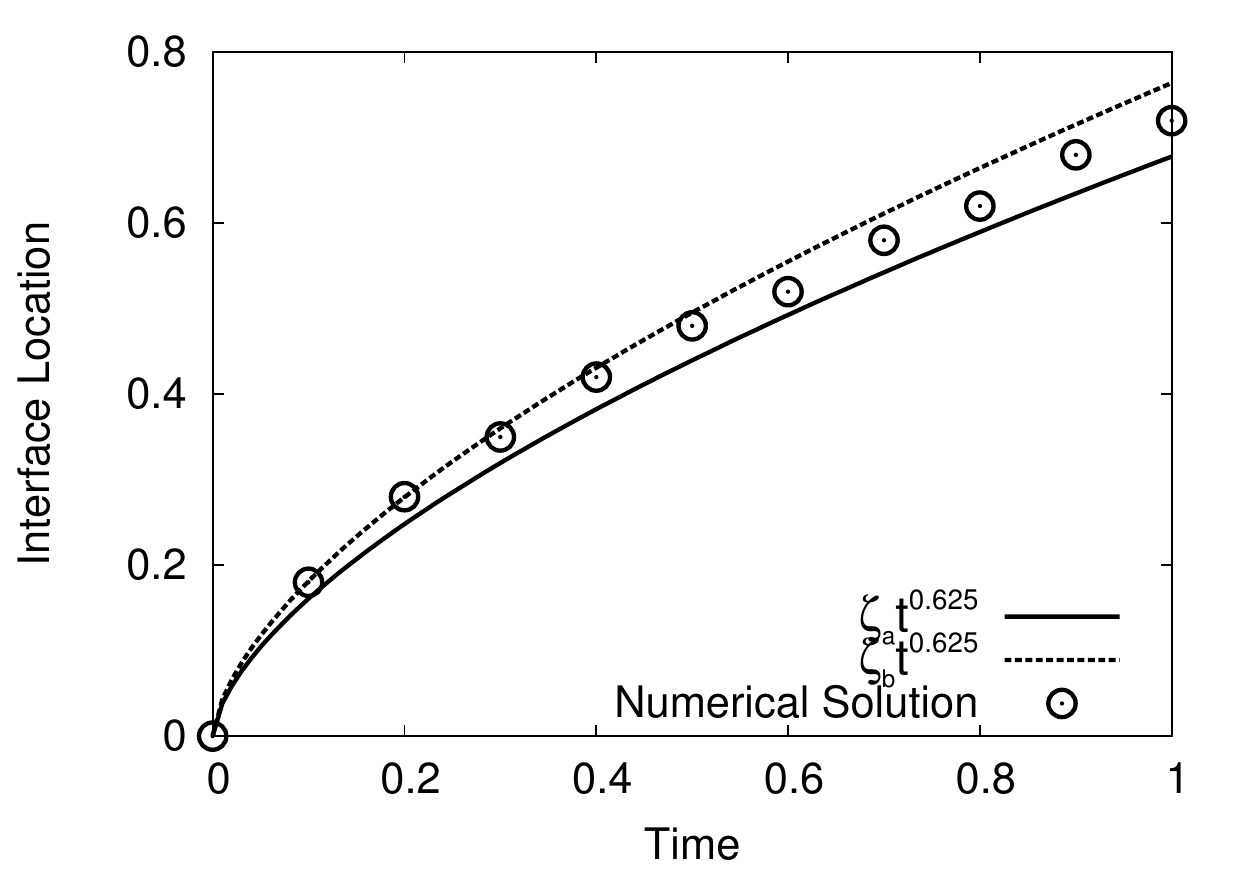}}
\caption{Region 1: Interface Location Vs. Time with (left) and without (right) the absorption term}
\label{fig:exp}
\end{figure}
\begin{figure}[tbhp]
\centering
\subfloat[$C = 0.5 > C_{*}$]{\includegraphics[width=0.43\textwidth]{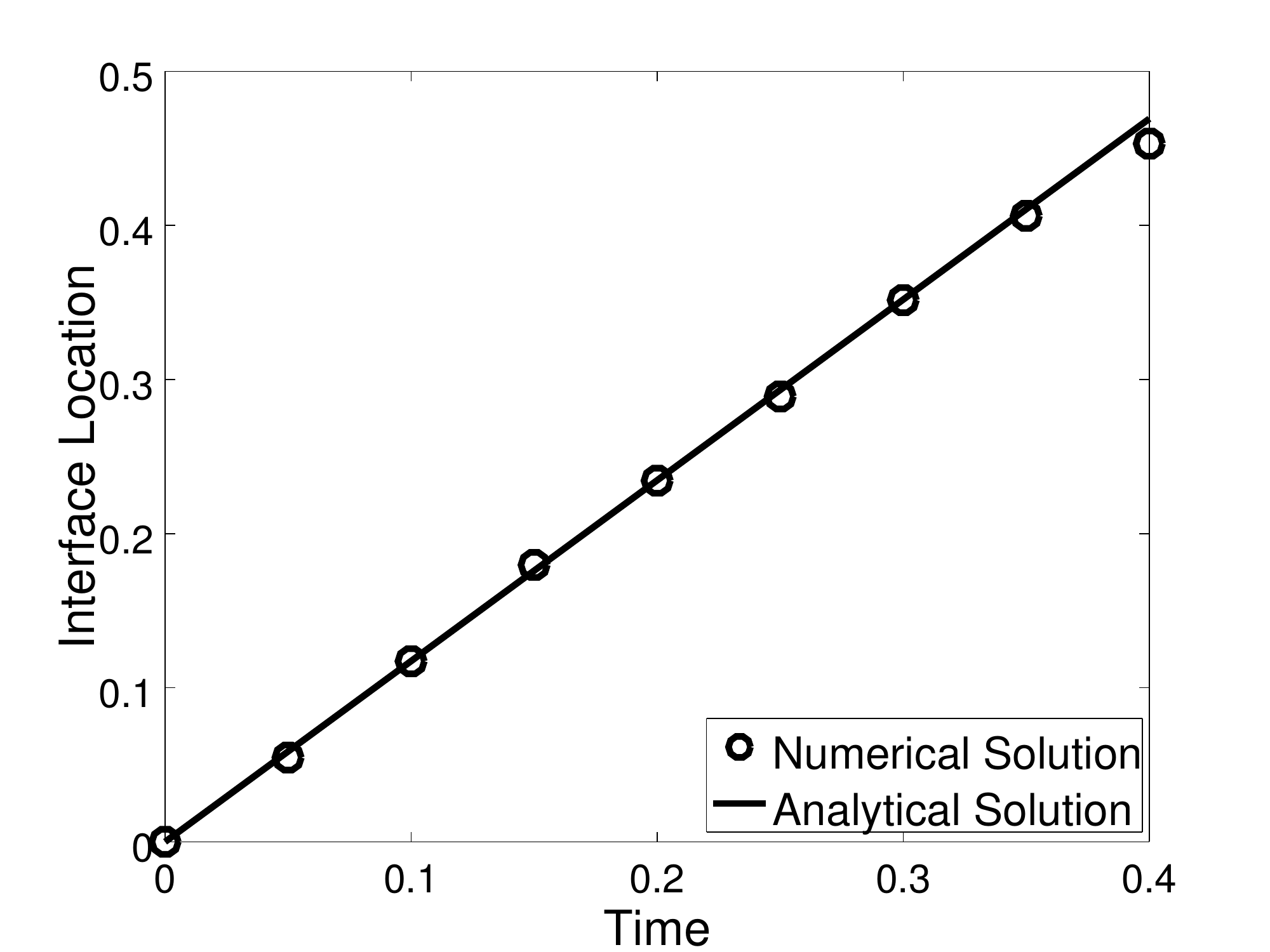}}
\subfloat[$C = 0.06 < C_{*}$]{\includegraphics[width=0.43\textwidth]{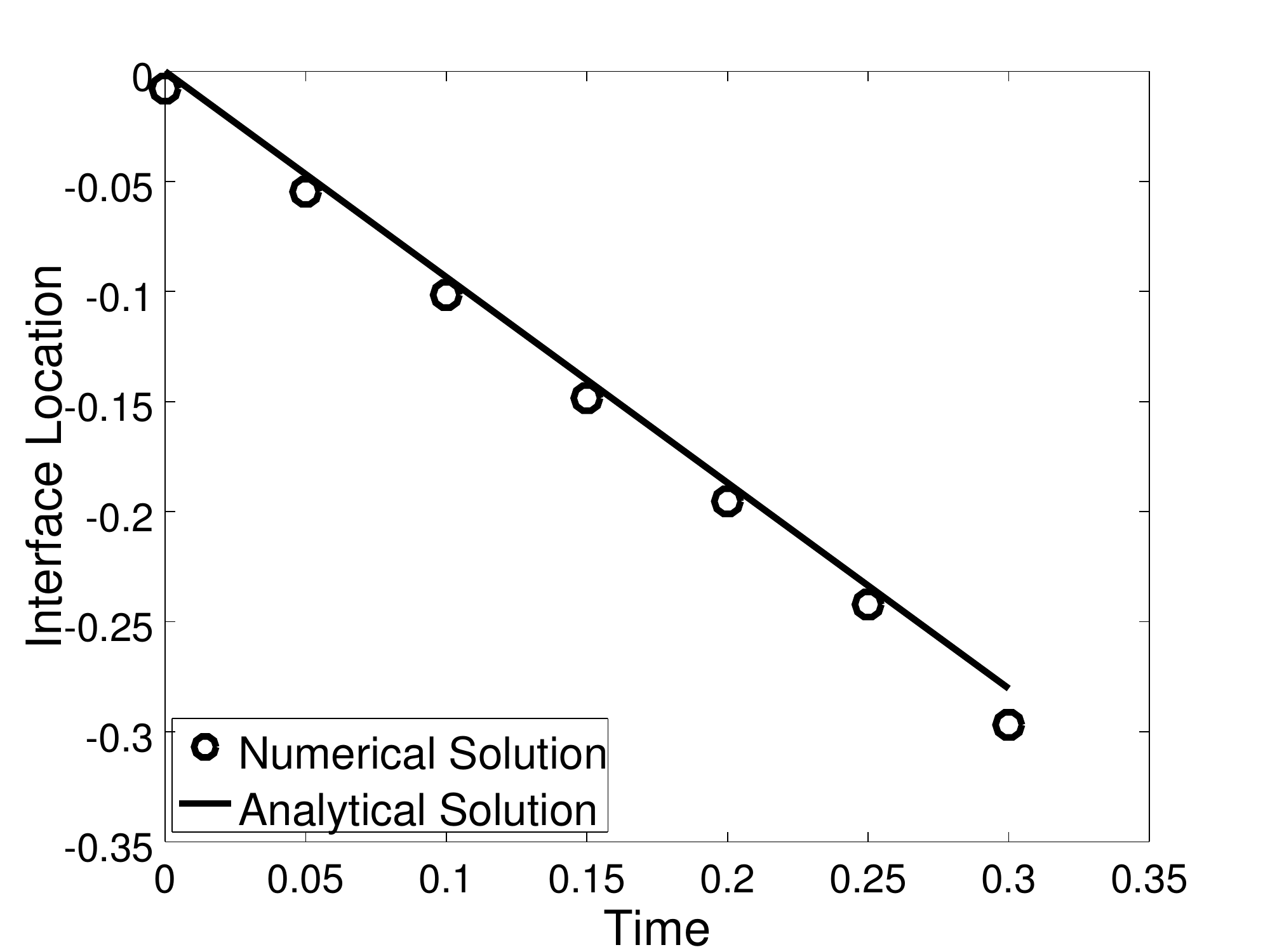}}
\caption{Region 2 with $p(m+\beta) = 1+p$. Interface Location Vs. Time for different $C$}
\label{fig:borderline1}
\end{figure}
\begin{figure}[tbhp]
\centering
\subfloat[$C = 1.2 > C_{*}$]{\includegraphics[width=0.43\textwidth]{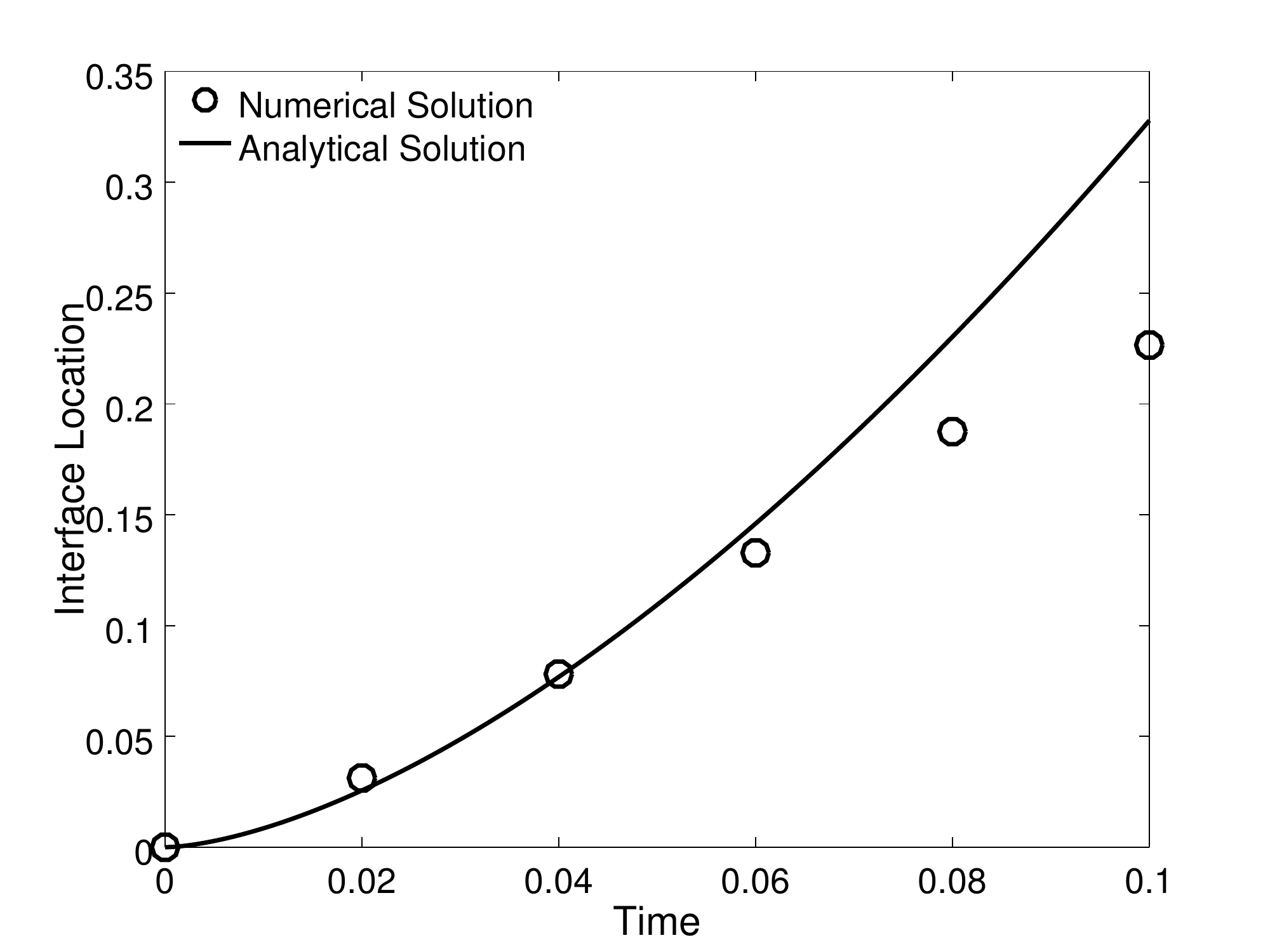}}
\subfloat[$C = 0.2 < C_{*}$]{\includegraphics[width=0.43\textwidth]{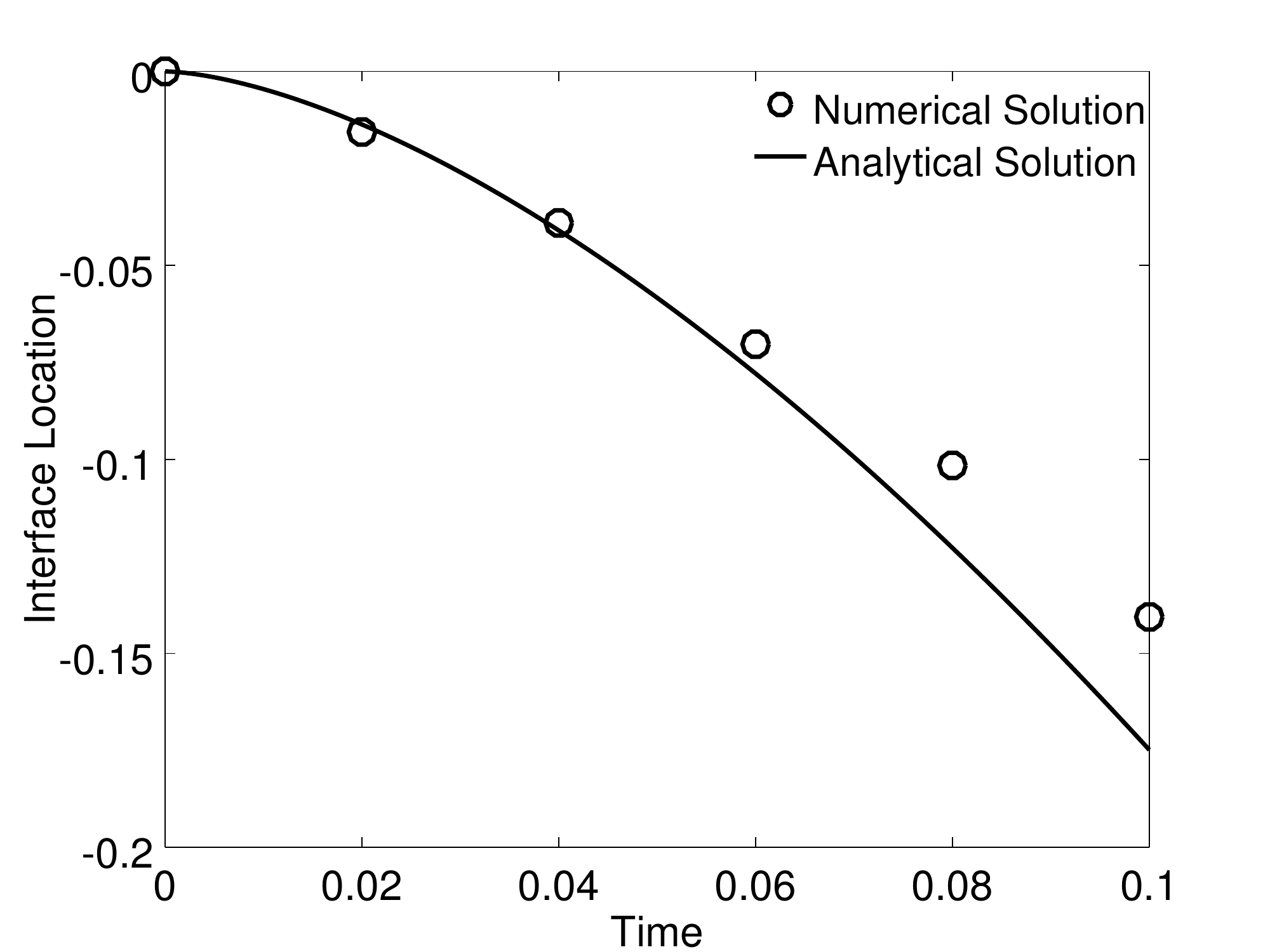}}
\caption{Region 2 with $p(m+\beta)> 1+p$. Interface Location Vs. Time for different $C$}
\label{fig:borderline2}
\end{figure}
\begin{figure}[tbhp]
\centering
\subfloat[$C = 0.4 > C_{*}$]{\includegraphics[width=0.43\textwidth]{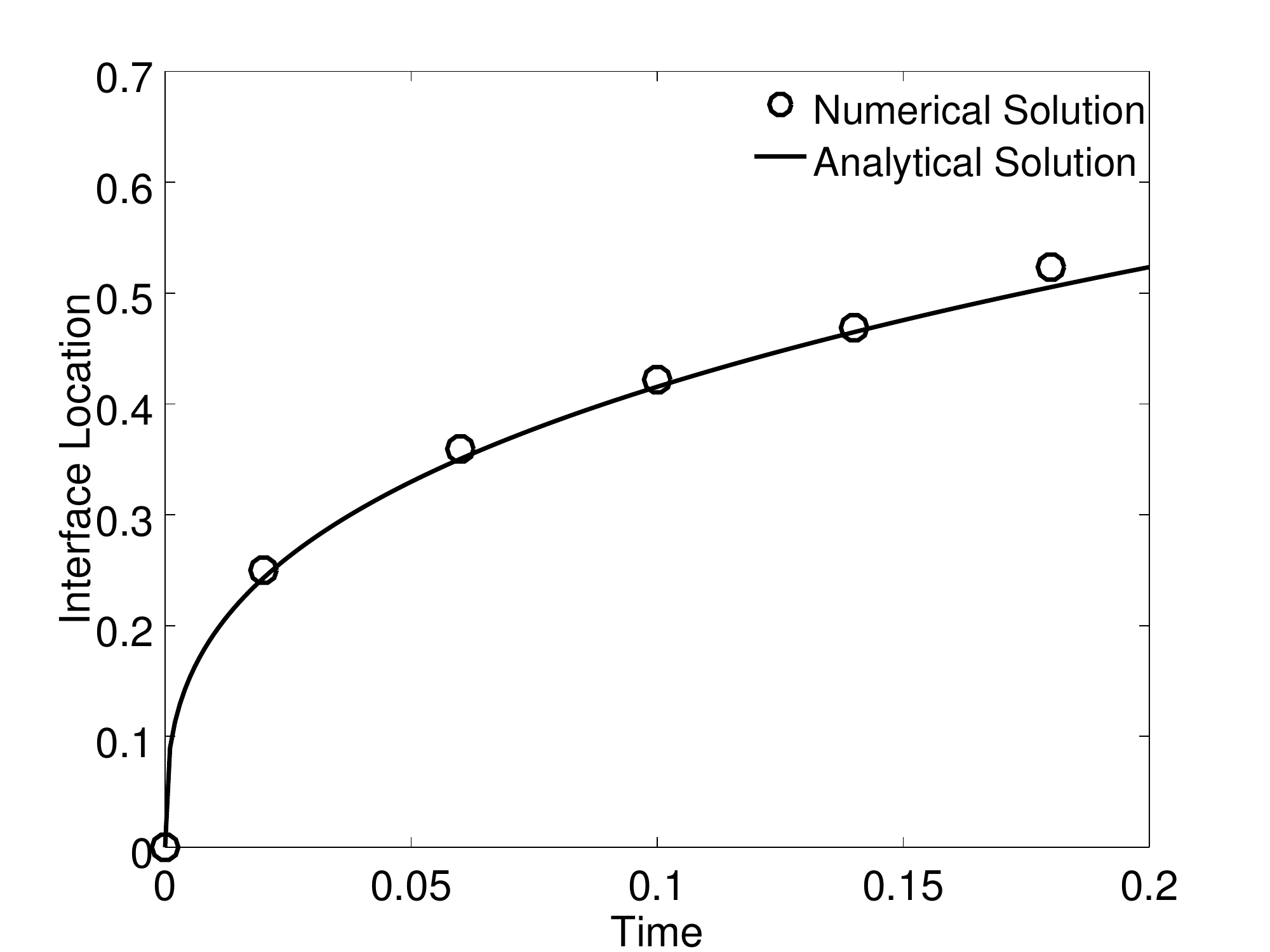}}
\subfloat[$C = 0.05 < C_{*}$]{\includegraphics[width=0.43\textwidth]{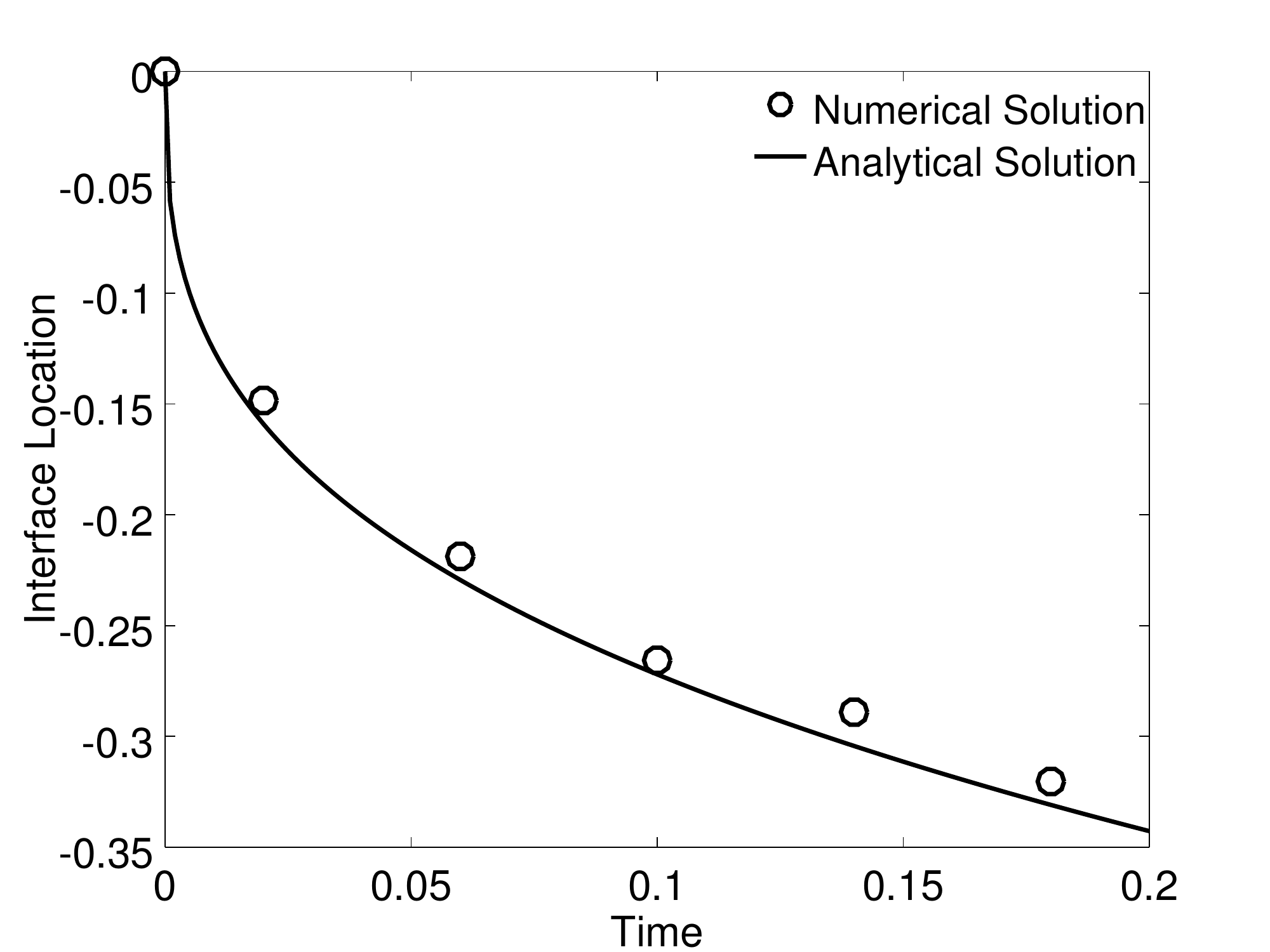}}
\caption{Region 2 with $p(m+\beta)< 1+p$. Interface Location Vs. Time for different $C$}
\label{fig:borderline3}
\end{figure}
\begin{figure}[tbhp]
\centering
\subfloat[$\eta(t)$ vs. Time]{\includegraphics[width=0.43\textwidth]{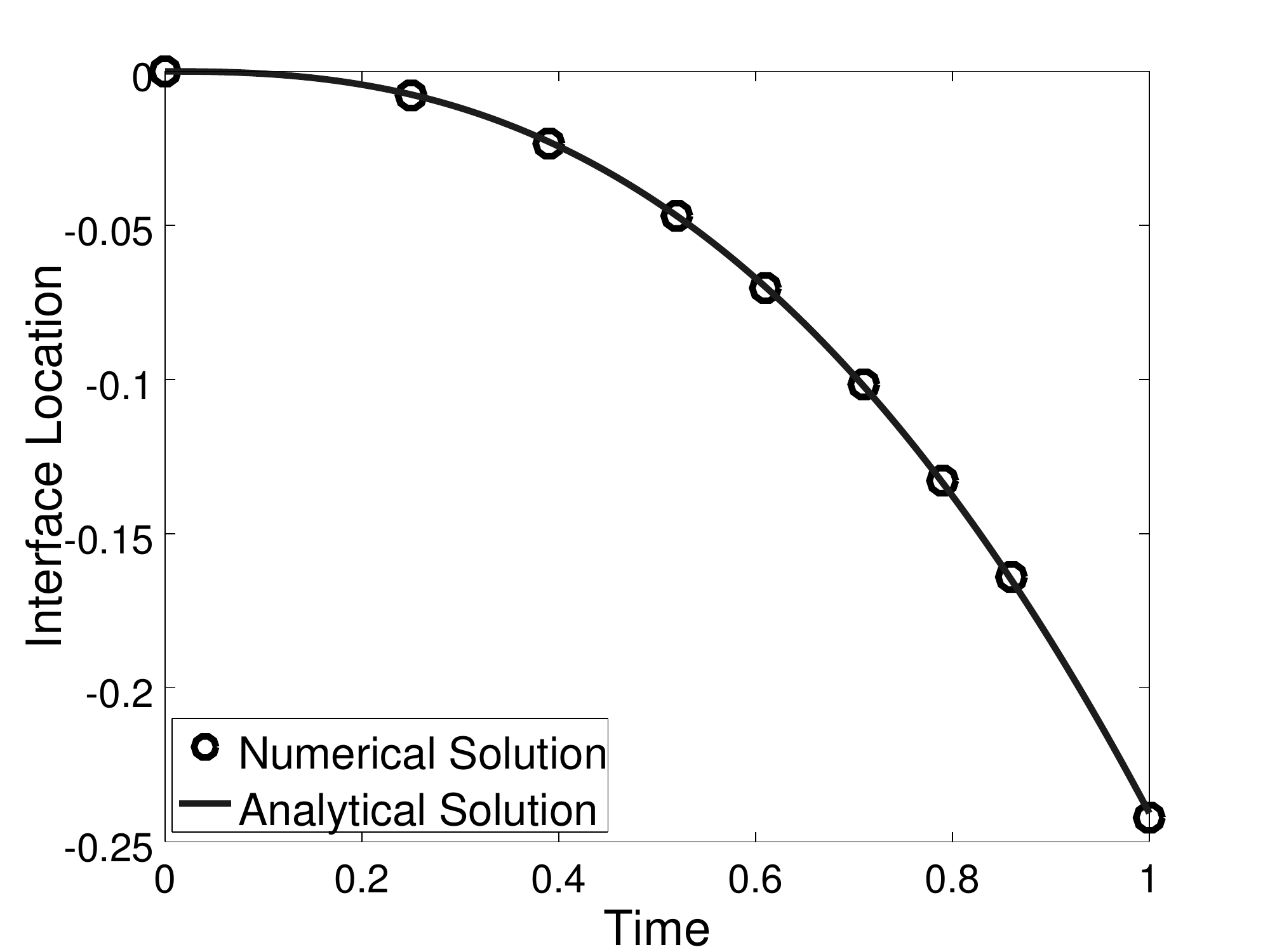}}
\subfloat[Solution at $t = 0.9$]{\includegraphics[width=0.43\textwidth]{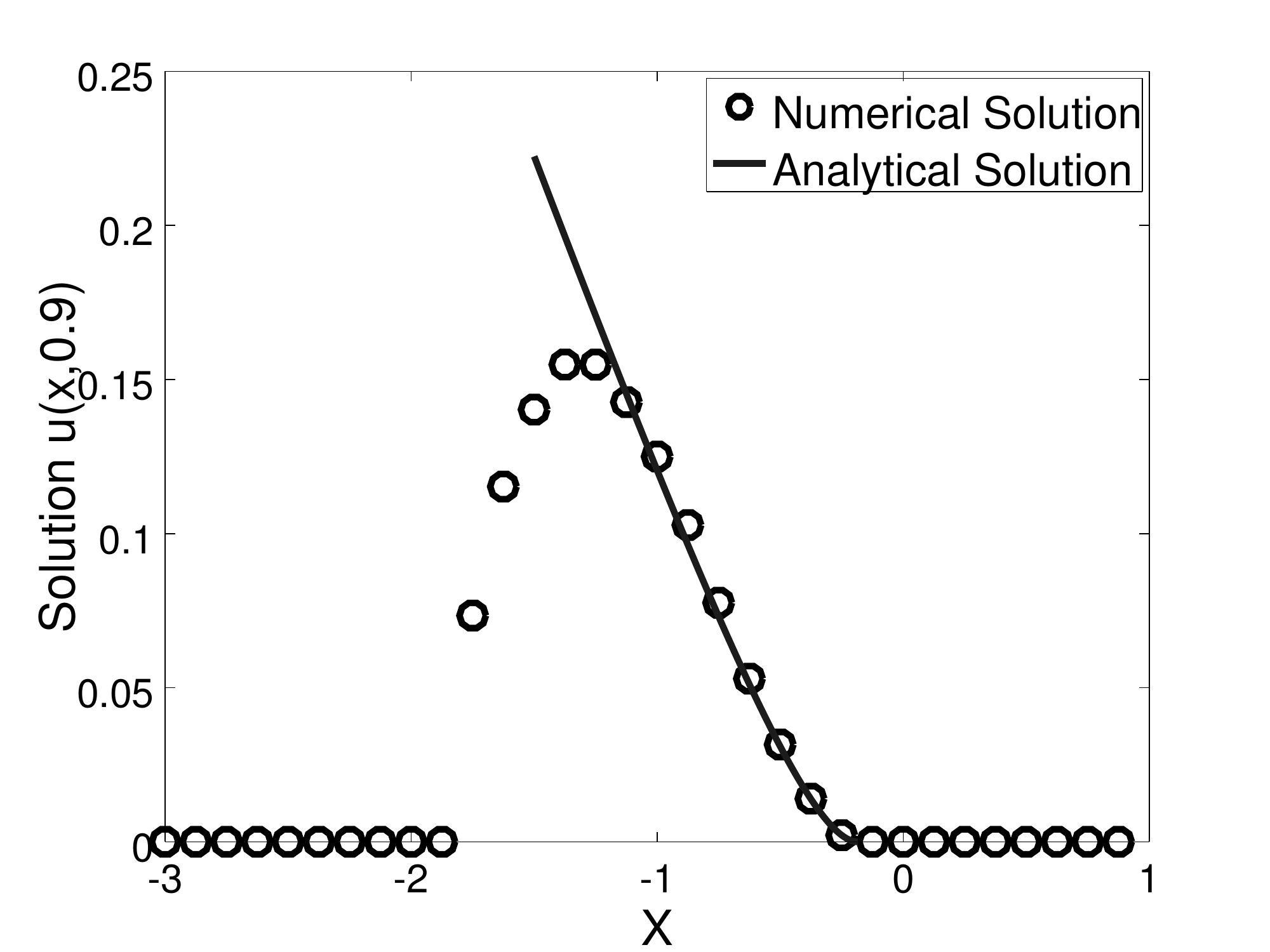}}
\caption{Region 3 Left: Interface Location Vs. Time; Right: Solution at $t=0.9$}
\label{fig:shrink}
\end{figure}
\begin{figure}[htbp]
\centering
\subfloat[Solution at different time]{\includegraphics[width=0.43\textwidth]{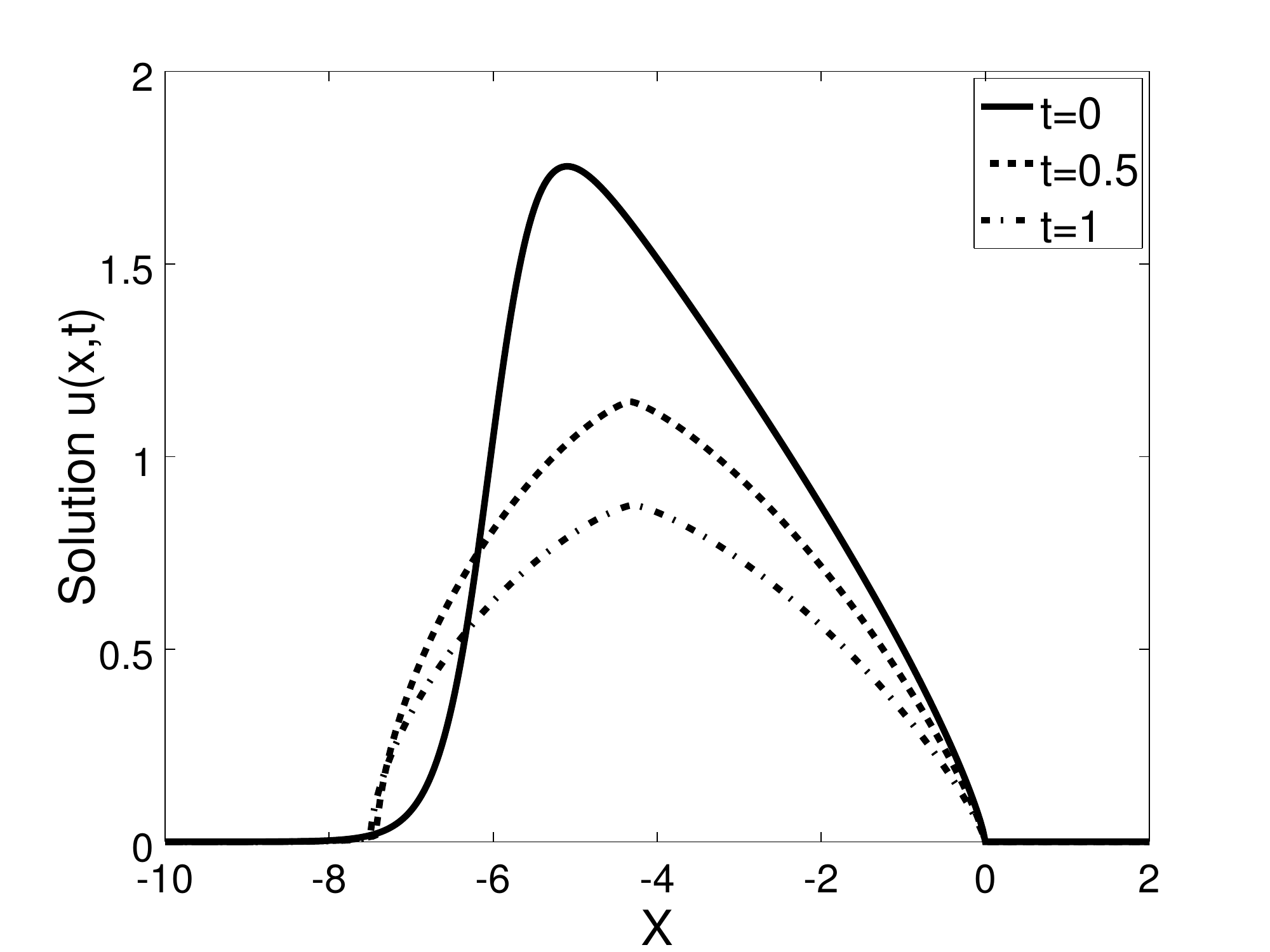}}
\subfloat[Solution at $t = 0.7$]{\includegraphics[width=0.43\textwidth]{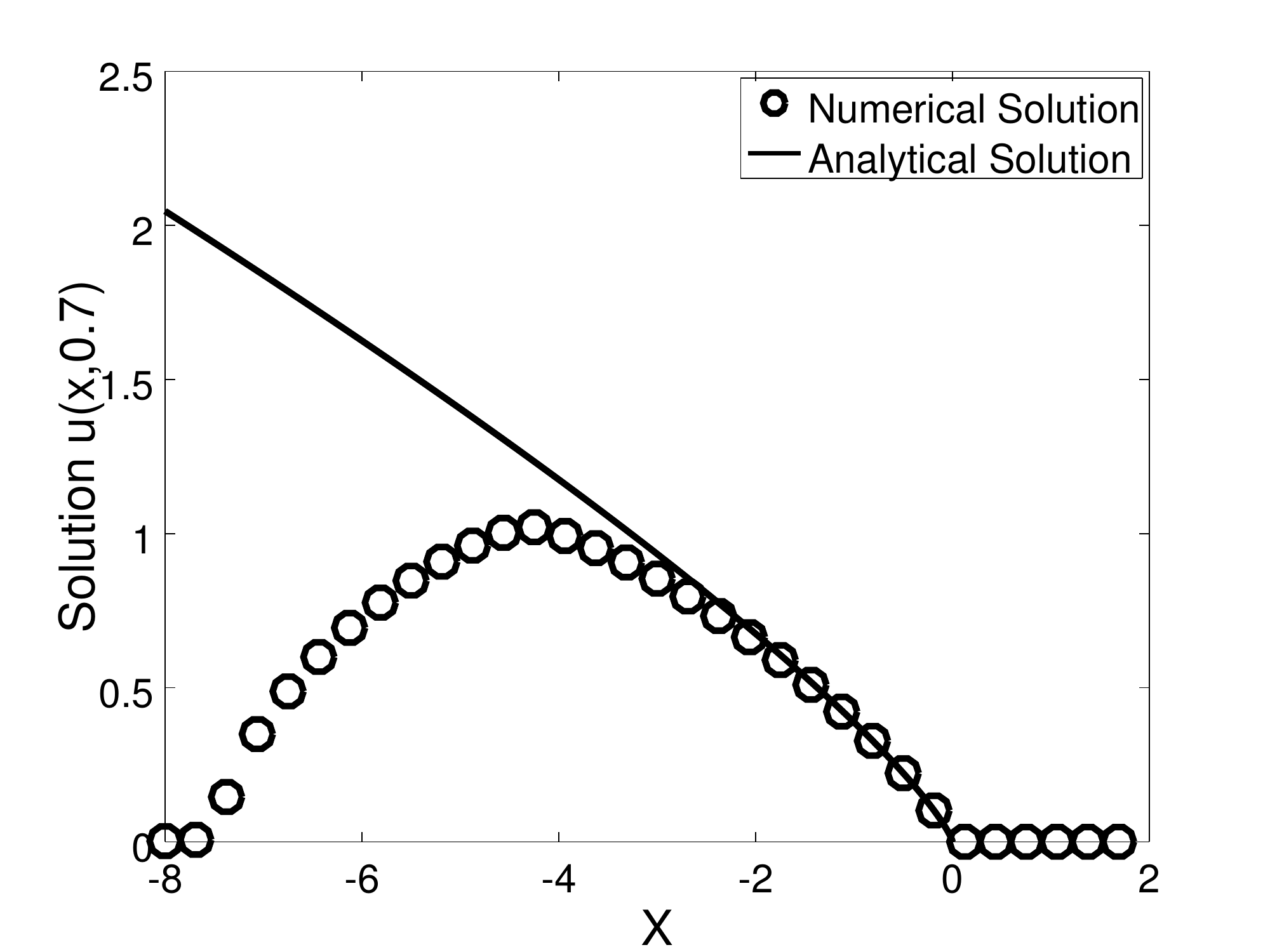}} \\
\caption{Region 4 Left: Numerical solution at different time; Right: Numerical solution at $t=0.7$}
\label{fig:case4a}
\end{figure}
\clearpage

	\bibliographystyle{plain}
	\bibliography{references}

\end{document}